\documentclass[12pt,authoryear]{elsarticle}

\usepackage{cc4elsart,ccalgo}
\usepackage[german,english]{babel}

\usepackage{natbib}
\usepackage{cyrillic}
\usepackage{url}

\usepackage{tilde}
\usepackage{maps}
\usepackage{mathell}
\usepackage{set}
\usepackage{xcolor}
\newcommand{\degOne}{{\color{red}l}}

\renewcommand{\degOne}{{\color{red}k}}

\renewcommand{\degOne}{k}

\newcommand{\C}{\mathbb{C}}
\newcommand{\F}{\mathbb{F}}

\newcommand{\nodiv}{\nmid}

\newcommand{\dg}{d}
\renewcommand{\dg}{\text{\textcolor{blue}{$d$}}}
\renewcommand{\dg}{n}

\newcommand{\parTwo}{{\color{red}CC}}
\renewcommand{\parTwo}{a}
\makeatletter
\newcounter{enumabc}
\def\theenumabc{\@alph\c@enumabc}
\def\labelenumabc{\theenumabc.}
\def\p@enumabc{.}

\newcounter{enumEsub}
\def\theenumEsub{\ensuremath{E_{\@arabic\c@enumEsub}}}
\def\labelenumEsub{\theenumEsub:}
\def\p@enumEsub{}

\makeatother

\begin{document}

\begin{frontmatter}
\title{Normal form for Ritt's Second Theorem}

\author{Joachim von~zur~Gathen}
\address{B-IT, Universit\"at Bonn\\
  D-53113 Bonn\\
http://cosec.bit.uni-bonn.de/}
\ead{gathen@bit.uni-bonn.de}

\begin{abstract}
  Ritt's Second Theorem deals with \emph{composition collisions}
  $g\circ h=g^{*}\circ h^{*}$ of univariate polynomials over a field,
  where $\deg g =\deg h^{*}$. Joseph Fels Ritt (1922) presented two
  types of such decompositions. His main result here is that these
  comprise all possibilities, up to some linear
  transformations. Because of these transformations, the result has
  been called ``difficult to use''. We present a normal form for
  Ritt's Second Theorem, which is hopefully ``easy to use'', and
  clarify the relation between the two types of examples. This yields
  an exact count of the number of such collisions in the ``tame
  case'', where the characteristic of the (finite) ground field does
  not divide the degree of the composed polynomial.

  A slightly abridged version of this paper has appeared in
  \emph{Finite Fields and Their Applications} {\bf 27} (2014), pages 41--71.
  \copyright\ 2013 Elsevier Inc.
\end{abstract}


% \begin{subject}
% 68W30, 11T06, 12E10
% \end{subject}

\end{frontmatter}

%%% \begin{enumabc}
%%% \item Test
%%% \item Blub
%%% \item Bla
%%% \end{enumabc}
%%% \begin{enumEsub}
%%% \item Test
%%% \item Blub
%%% \item Bla
%%% \end{enumEsub}

\section{Introduction}

In the 1920s, Ritt, Fatou, and Julia investigated the composition
\begin{equation}
\label{collision}
f=g\circ h=g(h)
\end{equation}
 of univariate polynomials over a field $F$ for $F =
\mathbb{C}$. It emerged as an important question to determine the
(distinct-degree) collisions (or nonuniqueness) of such decompositions, that is,
different components $(g,h)\neq(g^{*},h^{*})$ with equal composition
$g\circ h=g^{*}\circ h^{*}$ and equal sets of degrees: $\deg g = \deg
h^{*} \neq \deg h = \deg g^{*}$.

Composition with linear polynomials
(polynomials of degree 1) introduces inessential
ambiguities. Ritt presented two types of essential collisions:
\begin{align}\label{al:colli}
x^{l}\circ x^{k}w(x^{l}) & = x^{kl}w^{l}(x^{l})=x^{k}w^{l}\circ x^{l},\\
T_{m}(x,z^{l})\circ T_{l}(x,z) & = T_{lm}(x,z)=T_{l}(x,z^{m})\circ
T_{m}(x,z),\nonumber
\end{align}
where $w\in F[x]$, $z\in F^{\times} = F \smallsetminus \{0\}$, and $T_{m}$
is the $m$th Dickson polynomial of the first kind. And then he proved
that these are essentially all possibilities. Details are given below.

Ritt worked with $F=\mathbb{C}$ and used analytic
methods. Subsequently, his approach was replaced by algebraic methods,
in the work of \cite{lev42} and \cite{dorwha74}, and \cite{sch82c}
presented an elementary but long and involved argument. Thus Ritt's
Second Theorem was also shown to hold in positive characteristic
$p$. The original versions of this required $p>  n = \deg f = \deg(g\circ
h)$. \cite{zan93} reduced this to the milder and more natural
requirement $g'(g^{*})'\neq 0$. His proof works over an algebraically
closed field, and Schinzel's \citeyear{sch00c} monograph adapts it to
finite fields.
These results assume that $\gcd(\deg g, \deg g^*) = 1$;  \citet{tor88a}
removed this condition provided that $p \nmid n$.

Ritt's Second Theorem, stated as \ref{RST} below, involves four unspecified
linear functions, and it is not clear whether there is a relation
between the first and the second type of example. In particular, a
uniqueness property in Ritt's theorem is not obvious, and
indeed \cite{beang00} are puzzled by its absence. On their page 128,
they write, adapted to the notation of the present paper, ``Now these rules are
a little less transparent, and a little less independent, than may
appear at first sight. First, we note that [the First Case], which is
stated in its conventional form, is rather loosely defined, for the
$k$ and $w$ are not uniquely determined by the form $x^{k}w(x^{l})$;
for instance, if $w(0)=0$, we can equally well write this expression
in the form $x^{k+l}\tilde w(x^{l})$, where $\tilde w = w/x$. Next,
$T_{2}(x,1)=x^{2}-2$ differs by a linear component from $x^{2}$, so
that in some circumstances it is possible to apply [the Second Case]
to $T_{2}(x,1)$, then [a linear composition], and then (on what is
essentially the same factor) [the Second Case]. These observations
perhaps show why it is difficult to use Ritt's result.'' These
well-motivated concerns are settled by the result of the present
paper.

Similarly, \cite{bin95a}, Section 7, writes: ``This raises the
question, whether the two theorems of Ritt can be used to give a
general exact description of all possible decompositions. In
particular, we may ask whether there is a canonical decomposition.''
We answer this question for two components.

Namely, \ref{th:fifi} presents a normal form for the decompositions in
Ritt's Theorem. It makes Zannier's assumption $g'(g^{*})'\neq 0$ and
the standard assumption $~gcd(l,m)=1$, where $m=k+l\deg w$ in
\ref{al:colli}. This normal form is unique unless $p\mid m$.
 We also
elucidate the relation between the first and the second type of
example.

 A fundamental dichotomy in this business is
whether $p$ divides $\dg$ or not. The designation \emph{tame} and
\emph{wild} was introduced in \citet{gat90c,gat90d} for the cases
$p\nmid \dg$ and $p\mid\dg$, respectively, in analogy with
ramification indices.  An important consequence--and the original
motivation--of this normal form is that we can count exactly the number
of ``collisions'' as described by Ritt's Second Theorem (\ref{RST}),
over a finite field and in the tame case.
In turn, this is an essential ingredient for counting, approximately
or exactly, the decomposable polynomials over a finite field; see
\cite{gat14}.
Equal-degree collisions, where the degree conditions are replaced
by $\deg g = \deg g^*$, occur only in the wild case and
are not considered in this paper.

\ref{tab:cor-table} gives a pr\'ecis of these counting results. The
notation consists of a finite field $\mathbb{F}_{q}$ of characteristic
$p$, integers $l$ and $m$ with $m>l\geq 2$, $\dg=lm$,
the set $D_{n,l}$  of  monic
compositions of degree $n$ with constant coefficient $0$ and a left
component of degree $l$ (see \ref{eq:Zan}),
 $s=\lfloor
m/l\rfloor$, $c=\lceil(m-l+1)/l\rceil$,  $t=\#(D_{\dg,l}\cap
D_{\dg,m} \smallsetminus \mathbb{F}_{q}[x^{p}])$, and Kronecker's $\delta$.
\begin{table}
  \begin{center}
    \setcounter{enumi}{0}
    \begin{tabular}{l|l|l}
      &conditions&bounds on $t$\\\hline
     % \refstepcounter{enumi}\labelenumi\label{lem:div-1-table}
       (i)\Label{lem:div-1-table}{(i)}
      &$p\nmid n,\; ~gcd(l,m)=1$&$t=q^{s+1}+(1-\delta_{l,2})(q^{2}-q)$\\
      &&$q^{s+1}\leq t\leq q^{s+1}+q^{2}\leq 2q^{s+1}$\\
      %\refstepcounter{enumi}\labelenumi\label{lem:div-1/4-table}
       (ii)\Label{lem:div-1/4-table}{(ii)}
      &
      $p\mid l,\; ~gcd(l,m)=1$&$t=0$\\
      %\refstepcounter{enumi}\labelenumi\label{lem:div-1/3-table}
       (iii)\Label{lem:div-1/3-table}{(iii)}
      &$p\mid m,\;
      ~gcd(l,m)=1$&$t\leq q^{s+1}-q^{\lfloor s/p \rfloor +1}$\\
      %\refstepcounter{enumi}\labelenumi\label{thm:FFC-1-table}
       (iv)\Label{thm:FFC-1-table}{(iv)}
      &$p\nmid n,\; l\mid m$&$t=q^{2l+s-3}$\\
      %\refstepcounter{enumi}\labelenumi\label{thm:FFC-1a-table}
       (v)\Label{thm:FFC-1a-table}{(v)}
      &$p\nmid n,\; l\nmid m,\;
      ~gcd(l,m) = i$&$t=q^{2i}(q^{s-1}+(1-\delta_{l,2})(1-q^{-1}))$\\
      %\refstepcounter{enumi}\labelenumi\label{thm:FFC-2-table}
       (vi)\Label{thm:FFC-2-table}{(vi)}
      &$p\nmid n$&$t\leq q^{2l+s-3}$\\
      %\refstepcounter{enumi}\labelenumi\label{cor:ffchar-2-table}
       (vii)\Label{cor:ffchar-2-table}{(vii)}
      &$p\nmid l,\; p\mid m,$ &$t\leq q^{m+\lceil l/p\rceil-2}$\\
      %\refstepcounter{enumi}\labelenumi\label{cor:ffchar-3-table}
       (viii)\Label{cor:ffchar-3-table}{(viii)}
      &$p\mid l$ &$ t\leq q^{m+l-c+\lceil c/p\rceil-2}$
    \end{tabular}
  \end{center}
\caption{Bounds on $t$ in various cases.}\label{tab:cor-table}
\end{table}

%%
%%\begin{table}[ht!]
%%\begin{tabular}{l|l|l}
%%&conditions&bounds on $t$\\\hline
%%\emph{(i)}\Label{lem:div-1-table}{\emph{(i)}}&$p\nmid n,\; ~gcd(l,m)=1$&$t=q^{s+1}+(1-\delta_{l,2})(q^{2}-q)$\\
%%&&$q^{s+1}\leq t\leq q^{s+1}+q^{2}$\\
%%\emph{(ii)}\Label{lem:div-1/4-table}{\emph{(ii)}}&
%%$p\mid l,\; ~gcd(l,m)=1$&$t=0$\\
%%\emph{(iii)}\Label{lem:div-1/3-table}{\emph{(iii)}}&$p\mid m,\;
%%~gcd(l,m)=1$&$t\leq q^{s+1}-q^{\lfloor s/p \rfloor +1}$\\
%%\emph{(iv)}\Label{thm:FFC-1-table}{\emph{(iv)}}&$p\nmid n,\; l\mid m$&$t=q^{2l+s-3}$\\
%%\emph{(v)}\Label{thm:FFC-1a-table}{\emph{(v)}}&$p\nmid n,\; l\nmid m,\; ~gcd(l,m)=i$&$t=q^{2i}(q^{s-1}+(1-\delta_{l,2})(q^{2}-q))$\\
%%\emph{(vi)}\Label{thm:FFC-2-table}{\emph{(vi)}}&$p\nmid n$&$t\leq 2q^{2l+s-3}$\\
%%\emph{(vii)}\Label{cor:ffchar-2-table}{\emph{(vii)}}&$p\nmid l,\; pl\mid
%%m$&$t\leq q^{m+\lceil l/p\rceil-2}$\\
%%\emph{(viii)}\Label{cor:ffchar-3-table}{\emph{(viii)}}&$p=l\mid
%%m,\;c=\lceil(m-l+1)/l\rceil$&$t\leq q^{m+l-c+\lceil c/p\rceil-2}$\\
%%\end{tabular}
%%\caption{Bounds on $t$ in various cases.}\label{tab:cor-table}
%%\end{table}
%%%%% Ende einfügen von tab:cor-table aus count4

The basic normal form result is augmented in several
directions. Firstly, we can relinquish the condition $g'(g^{*})'\neq
0$, keeping the assumption $~gcd(l,m)=1$ (\ref{lem:div2a}).  For the
standard form of Ritt's Second Theorem, this is already noted in
\cite{zan93}. Secondly, based on a result of \citet{tor88a}, we can
allow $l$ and $m$ to have a nontrivial common divisor, provided that
$p\nmid \dg$ (\ref{thm:FFC}). Finally, if $p\mid\dg$ and
$~gcd(l,m)>1$, we use derivatives to obtain restrictions on the
decompositions (\ref{thm:div2a}); in contrast to the previous results,
these are not expected to be sharp.

The present results are of independent interest, but they arose in a
larger context. Leonard Carlitz and later Stephen Cohen and others
derived estimates for the number of reducible multivariate
polynomials; see \citet{gatvio13} for more history and
references. The latter paper also contains results on squareful and
relatively irreducible (that is, irreducible and not absolutely
irreducible)
multivariate polynomials. Fairly precise estimates for decomposable
polynomials are known in the multivariate \citep{gat10f}
 %\citep{gat08b}
and the univariate \citep{gat14} scenarios, but with less precision in the
latter case in some special ``wild'' situations, namely when the
characteristic is the smallest prime divisor of $\dg$ and divides it
exactly twice. The counting results of the present paper form a
cornerstone for those univariate estimates, and were actually derived
in the context of that work. \citet*{blagat13} determine the number of
decomposables of degree $p^{2}$. \citet*{boddeb09} also deal with counting.
On a different but related topic, \cite{ziemue08} found interesting
characterizations for Ritt's First Theorem, which deals with complete
decompositions, where all components are indecomposable.

\section{Decompositions}\label{secDec}

\begin{definition}
  \label{defComp}
  For $g, h \in F[x]$,  $ f = g \circ h = g(h) \in F[x]$
  is their \emph{composition}.  If $\deg g, \deg h \geq 2$, then
  $(g,h)$ is a \emph{decomposition} of $f$. A polynomial $f \in F[x]$
  of degree at least 2
  is \emph{decomposable} if there exist such $g$ and $h$, otherwise
  $f$ is \emph{indecomposable}.
\end{definition}
In the literature, $(g,h)$ is sometimes called a
\emph{bidecomposition} of $f$.
A nonzero polynomial $f\in F[x]$ over a field $F$ is \emph{monic} if
its leading coefficient $~lc(f)$ equals $1$.  We call $f$
\emph{original} if its graph contains the origin, that is, $f(0)=0$.

\label{invariance1}
  Multiplication by a unit or addition of a constant does not change
  decomposability, since
  $$
  f = g \circ h \Longleftrightarrow a f+b = (a g+b) \circ h
  $$
  for all $f$, $g$, $h$ as above and $a,b \in F$ with $a\neq 0$.  In
  other words, the set of decomposable polynomials is invariant under
  this action of $F^{\times} \times F$ on $F[x]$,
    and the linear polynomials are the units
  for the composition operation.
In particular, if we have a set $S$ of monic original composable
polynomials and let $S^{*}$ be the set of all their compositions with a
linear component on the left, then
\begin{equation}\label{eq:linfac}
\#S^{*} = q^{2}(1-q^{-1}) \cdot \#S.
\end{equation}

Furthermore, any decomposition $(g,h)$ can be normalized by this
action, by taking $a = ~lc (h)^{-1} \in F^{\times}$, $b=-a \cdot h(0)
\in F$, $g^{*} = g((x-b)a^{-1}) \in F[x]$, and $h^{*} = ah+b$.  Then
$g\circ h = g^{*} \circ h^{*}$ and $g^{*}$ and $ h^{*}$ are monic original.

It is therefore sufficient to consider compositions $f = g \circ h$
where all three polynomials are monic original. Then $(g,h)$ is called
monic original. If $D_{\dg}$ is the set of such $f$ of degree $\dg$,
then the number of all decomposable polynomials of degree $\dg$, not
restricted to monic original, is
\begin{equation}\label{eq:notrestr}
q^{2}(1-q^{-1}) \cdot \#D_{\dg}
\end{equation}

We fix some notation for the remainder of this paper. $F$ is a field
of characteristic $p\geq 0$. For $\dg \geq
1$, we write
\begin{align*}
  P_{\dg} & = \{f \in F[x] \colon \deg f = \dg, f\text{ monic original}\}.
\end{align*}
A decomposition \ref{collision} is \emph{tame} if $p\nmid ~deg g$, and
$f$ is \emph{tame} if $p\nmid \dg$. It is well known that in a tame
decomposition, $g$ and $h$ are uniquely determined by $f$ and $\deg g$.
For any proper divisor
$e$ of $\dg$, we have the composition map
$$
\map[\gamma_{\dg,e}] {P_{e} \times P_{\dg/e}} {P_{\dg}}
{(g,h)} {g \circ h,}
$$
corresponding to \ref{collision}, and set
\begin{equation}\label{eq:Zan}
  D_{\dg,e}= ~im \gamma_{\dg,e}.
\end{equation}

The set $D_{\dg}$ of all decomposable polynomials in $P_{\dg}$
satisfies
\begin{equation}\label{substack}
  D_{\dg}= \bigcup_{\substack{e\mid \dg\\1<e<\dg}} D_{\dg,e}.
\end{equation}
In particular, $D_{\dg} = \varnothing$ if $\dg$ is prime. Our
collisions turn up in the resulting inclusion-exclusion formula for
$\# D_{\dg}$. Over a finite field $\mathbb{F}_{q}$ with $q$ elements,
we have
\begin{align*}
  \#P_{\dg} &= q^{\dg-1},\\
  \#D_{\dg,e}  &\leq q^{e+n/e-2}.
\end{align*}

It is useful to single out a special case of wild compositions.
\begin{definition}\label{rem:coll}
Over a field $F$ of characteristic $p > 0$,
  we call \emph{Frobenius composition} any $f \in F[x^{p}]$, since
  then $f=x^{p} \circ h^{*}$ for some $h^{*} \in P_{\dg/p}$, and any
  decomposition $(g,h)$ of $f = g \circ h$ is a \emph{Frobenius
    decomposition}. For any integer $j$, we denote by $\varphi_{j}
  \colon F \longrightarrow F$ the $j$th power of the Frobenius
  map  with
  $\varphi_{j}(a)= a^{p^{j}}$ for all $a \in F$, and extend it to an
  $\mathbb{F}_{p}$-linear map $\varphi_{j} \colon F[x]
  \longrightarrow F[x]$ with $\varphi_{j}(x)=x$. Then if $h \in F[x]$,
  we have
  \begin{equation}\label{eq:frob}
    x^{p^{j}} \circ h = \varphi_{j} (h) \circ x^{p^{j}}.
  \end{equation}
\end{definition}

Thus any Frobenius composition except $x^{p^{2j}}$ is the result of a
collision. Over $F = \mathbb{F}_{q}$, there are $q^{p^{j}-1}-1$ many
$h \in P_{p^{j}}$ with $h \neq x^{p^{j}}$ and for $m \neq p^{j}$, this
produces $q^{m-1}$ collisions with $h \in P_{m}$. This example is
noted in \cite{sch82c}, Section I.5, page 39.

\begin{example}\label{ex:sl}
  We look at decompositions $(g,h)$ of univariate quartic polynomials
  $f$, so that $\dg=4$. By \ref{invariance1}, we may assume $f\in
  P_{4}$, and then also $g$ is monic original. Thus the general case
  is
  $$
  (x^{2}+ax) \circ (x^{2}+bx) = x^{4} +ux^{3} + vx^{2} + wx \in F[x],
  $$
  with $a,b,u,v,w \in F$.  We find that with $a=2w/u$ and $b=u/2$
  (assuming $2u \neq 0$), the cubic and linear coefficients match, and
  the whole decomposition does if and only if
  $$u^{3} -4uv + 8w =0.$$
  This is a defining equation for the $2$-dimensional hypersurface of
  decomposable polynomials in $P_{4}$ (if $~char F \neq
  2$). This example is also in \cite{barzip76,barzip85}.
  %  $u = 0 \& 2 \neq 0 \Rightarrow b=w=0, a=v$. This is a line on the hypersurface.
% ~char = 2 \Rightarrow u=0, a=b^2+v, b^3+vb+w=0$. This is a general cubic.
% v=w=1$: no solution in $F_2$.
\end{example}

For $f \in P_{n}$ and $a \in F$, the (original) \emph{shift} of $f$ by
$a$ is
$$
f^{[a]}= (x - f(a)) \circ f \circ (x+a).
$$
Then $f^{[a]}$ is again monic original. This defines an action of the
additive group $F$ on $P_{n}$. It respects composition:
$$
f = g \circ h \Longleftrightarrow f^{[a]} = g^{[h(a)]} \circ h^{[a]}.
$$
We write $(g,h)^{[a]} = (g^{[h(a)]}, h^{[a]})$ for this decompositon
of $f^{[a]}$.

\section{Normal form: nonvanishing derivatives and coprime
  degrees}\label{sec:collcomp}

In this section, we treat the most basic (and most important) case of
(distinct-degree) collisions, namely, where the two component degrees
are coprime and the left components  have a nonvanishing
derivative. The latter is always satisfied in the tame case.
The following is an example of a collision:
$$
x^{k}w^{l} \circ x^{l}= x^{kl}w^{l}(x^{l})= x^{l} \circ x^{k}w(x^{l}),
$$
for any polynomial $w \in F[x]$, where $F$ is a field (or even a
ring).  We define the (bivariate) \emph{Dickson polynomials of the
first kind} $T_{m} \in F[x,y]$ by $T_{0}=2$, $T_{1}=x$,
and
\begin{equation}\label{Trecursion}
  T_{m}=xT_{m-1}- yT_{m-2}
  \text{ for }m \geq 2.
\end{equation}
The monograph of \cite{lidmul93a} provides extensive information about
these polynomials. We have $T_{m}(x,0)= x^{m}$, and $T_{m}(x,1)$ is
closely related (over $\mathbb R$) to the \emph{Chebyshev polynomial} $C_{\dg}=\cos(\dg
\arccos x)$, as $T_{\dg}(2x,1)=2C_{\dg}(x)$. $T_{m}$ is monic (for $m
\geq 1$) of degree $m$, and
$$
T_{m}= \sum_{0 \leq i \leq
  m/2}\frac{m}{m-i}\binom{m-i}{i} x^{m-2i} (-y)^{i} \in F[x,y].
$$
Furthermore,
\begin{equation}\label{eq:Trecursion}
T_{m}(x,y^{l})\circ T_{l}(x,y)=
T_{lm}(x,y)=T_{l}(x,y^{m})\circ T_{m}(x,y),
\end{equation}
and if $l \neq m$, then substituting any $z \in F$ for $y$ yields a
collision. (Here and in the previous example, the components are not
necessarily original.)

Ritt's Second Theorem is the central tool for understanding
distinct-degree collisions, and the following notions enter the
scene. The functional inverse $v^{-1}$ of a linear polynomial
$v=ax+b$ with $a$, $b \in F$ and $a \neq 0$ is defined as
$v^{-1}=(x-b)/a$.  Then $v^{-1} \circ v = v \circ v^{-1} =
x$. Two pairs $(g,h)$ and $(g^{*},h^{*})$ of polynomials are called
\emph{equivalent} if there exists a linear polynomial $v$ such that
$$
g^{*}=g \circ v, \ h^{*}=v^{-1} \circ h.
$$
Then $g \circ h = g^{*}\circ h^{*}$, and we write $(g,h) \sim
(g^{*},h^{*})$. The following result of \cite{rit22} (for $F=\C$) says
that, under certain conditions, the examples above are essentially the
only distinct-degree collisions. We use the strong version of
\cite{zan93}, adapted to finite fields. The adaption uses
\cite{sch00c}, Section 1.4, Lemma 2, and leads to his Theorem
8. Further references can be found in this monograph as well.
\begin{fact}\label{RST}(Ritt's Second Theorem)
  Let $l$ and $m$ be integers, $F$ a field, and $g$, $h$, $g^{*}$,
  $h^{*}\in F[x]$ with
  \begin{equation}\label{eq:RST-1}
    m > l \geq 2, \gcd(l,m)=1, \deg g= \deg h^{*} = m,\\
    \deg h = \deg g^{*} = l,
  \end{equation}
  \begin{equation}\label{eq:RST-2}
    g'(g^{*})' \neq 0,
  \end{equation}
  where $g'= \partial g/ \partial x$ is the derivative of $g$. Then
  \begin{equation}\label{eq:RST-3}
    g \circ h = g^{*} \circ h^{*}
  \end{equation}
  if and only if
  $$
  \exists k \in \mathbb{N}, v_{1}, v_{2} \in F[x] ~\text{linear}, w
  \in F[x] ~\text{with}~ k+l \deg w = m, z\in F^{\times},
  $$
  so that either
  \begin{equation}
    \tag*{First Case}
    \left\{
      \begin{aligned}
        (v_{1} \circ g, h \circ v_{2})  &\sim (x^{k} w^{l}, x^{l}),\\
        (v_{1} \circ g^{*}, h^{*} \circ v_{2}) &\sim (x^{l},
        x^{k}w(x^{l})),
      \end{aligned}
    \right.
  \end{equation}
  or
  \begin{equation}
    \tag*{Second Case}
    \left\{
      \begin{aligned}
        (v_{1} \circ g, h \circ v_{2})&\sim ( T_{m}(x,z^{l}), T_{l}(x,z)),\\
        (v_{1} \circ g^{*}, h^{*} \circ v_{2}) &\sim (T_{l}(x,z^{m}), T_{m}(x,z)).
      \end{aligned}
    \right.
  \end{equation}
\end{fact}

We have arranged collisions \ref{eq:RST-3} so that $\deg g > \deg g^*$.
The conclusion of the First Case is asymmetric in $l$ and $m$, but not in
the Second Case.

According to \ref{secDec}, we may also assume the following.
\begin{equation}\label{eq:intersec}
  f = g \circ h,\text{ and } g, h, g^{*}, h^{*} \text{ are monic original}.
\end{equation}

The following lemma about Dickson polynomials will be useful for our
normal form. We write $T_{\dg}'(x,y)=\partial T_{\dg}(x,y)/\partial x$
for the derivative with respect to $x$.
\begin{lemma}\label{Tsqfree}
  Let $F$ be a field of characteristic $p \geq 0$, $n \geq 1$, and $z
  \in F^{\times}$.
  \begin{enumerate}
  \item
    \label{Tsqfree-1}
    If $p=0$, or $p \geq 3$ and $\gcd(n,p)=1$, then the derivative
    $T'_n(x,z)$ is squarefree in $F[x]$.
  \item
    \label{Tsqfree-1b}
    If $p=0$ or $\gcd(n,p)=1$, and $n$ is odd, then there exists some
    monic squarefree $u\in F[x]$ of degree $(n-1)/2$ so that
    $T_n(x,z^{2})=(x-2z) \cdot u^{2}+2z^{\dg}$.
  \item
    \label{Tsqfree-2}
    Let $\gamma = (-y)^{\lfloor n/2 \rfloor}$.  $T_n$ is an odd
    or even polynomial in $x$ if $n$ is odd or even, respectively.  It has
    the form
    \begin{align*}
    T_n =
    \begin{cases}
      x^n - nyx^{n-2} +- \cdots + \gamma \dg x & \text{if }n\text{ is odd},\\
      x^n - nyx^{n-2} +- \cdots + 2\gamma & \text{if }n\text{ is even}.
    \end{cases}
    \end{align*}
  \item
    \label{Tsqfree-3}
    If $p \geq 2$, then $T_{p^j} = x^{p^j}$ for $j\geq 0$.
  \item \label{Tsqfree-3b} If $p \geq 2$ and $p \mid \dg$, then
    $T_{\dg}'=0$.
  \item
    \label{Tsqfree-4}
    For a new indeterminate $t$, we have
    $t^{\dg}T_{\dg}(x,y)=T_{\dg}(tx,t^{2}y)$.
  \item
    \label{Tsqfree-5}
    $T_{\dg}(2z,z^{2})=2z^{\dg}$.
  \end{enumerate}
\end{lemma}
\begin{proof}
  \ref{Tsqfree-1} \cite{wil71} and Corollary 3.14 of \cite{lidmul93a}
  show that if $F$ contains a primitive $n$th root of unity $\rho$,
  then $T'_n(x,z)/nc$ factors over $F$ completely into a product of
  quadratic polynomials $(x^{2}-\alpha_{k}^{2}z)$, where $1\leq k
  <\dg/2$, the $\alpha_k = \rho^k+ \rho^{-k}$ are Gau\ss\ periods
  derived from $\rho$, and the $\alpha_k^2$ are pairwise distinct,
  with $c=1$ if $n$ is odd and $c=x$ otherwise. We note that
  $\alpha_{k}=\alpha_{\dg-k}$.  We take an extension field $E$ of $F$ that
  contains a primitive $n$th root of unity and a square root $z_{0}$
  of $z$. This is possible since $p=0$ or $\gcd(n,p)=1$.%%% If $F$ is
 %%% finite with $q$ elements, then an extension of degree equal to the
 %%% order $e$ of $q$ modulo $n$ is sufficient, since then $n \mid
 %%% q^e-1$.
  ~Thus $x^{2}- \alpha_{k}^{2}z = (x- \alpha_{k}z_{0})(x+
  \alpha_{k}z_{0})$, and the $\pm \alpha_{k}z_{0}$ for $1 \leq k < n/2$ are
  pairwise distinct, using that $p \neq 2$. It follows that
  $T'_n(x,z)$ is squarefree over $E$. Since squarefreeness is a
  rational condition, equivalent to the nonvanishing of the
  discriminant, $T'_n(x,z)$ is also squarefree over $F$.

  For \ref{Tsqfree-1b}, we take a Galois extension field~$E$ of~$F$
  that contains a primitive $n$th root of unity $\rho$, and set
  $\alpha_k = \rho^k + \rho^{-k}$ and $\beta_k = \rho^{k} - \rho^{-k}$
  for all $k\in\mathbb{Z}$.  We have $T_{n}(2z,z^{2}) = 2z^{\dg}$ by
  \ref{Tsqfree-5}, proven below, and Theorem~3.12(i) of
  \cite{lidmul93} states that
  $$
   T_{n}(x,z^{2})-2z^{\dg} = (x-2z) \prod_{1 \leq k < n/2}
   (x^{2}-2\alpha_{k}zx+4z^{2}+\beta_{k}^{2}z^{2});
  $$
  see also \cite{tur95}, Proposition 1.7. Now $ -\alpha_{k}^{2} + 4 +
  \beta_{k}^{2} = -(\rho^{k}+\rho^{-k})^{2} + (\rho^{k}-\rho^{-k})^{2}
  + 4 = 0, $ so that $ x^{2}-2 \alpha_{k}z x+4z^{2}+\beta_{k}^{2}z^{2}
  = (x-\alpha_{k}z)^{2}$. We set $u= \prod_{1 \leq k < n/2} (x-
  \alpha_{k}z)\in E[x]$. Then $T_{n}(x,z^{2})-2z^{\dg} = (x-2z)u^{2}$,
  and $u$ is squarefree. It remains to show that $u \in F[x]$. We take
  some $\sigma \in~Gal(E:F)$.  Then $\sigma(\rho)$ is also a primitive
  $\dg$th root of unity, say $\sigma(\rho)=\rho^{i}$ with $1\leq
  i<\dg$ and $~gcd(i,\dg)=1$. We take some $k$ with $1\leq k<\dg/2$,
  and $j$ with $ik\equiv j \bmod \dg$ and $0<|j| <\dg/2$. Then
  $\sigma(\alpha_{k})=\alpha_{|j|}$. Hence, $\sigma$ induces a
  permutation on $\Set{\alpha_{1}, \dots, \alpha_{(n-1)/2}}$. It
  follows that

  $$u = \prod_{1 \leq k < n/2} (x-\alpha_{k}z) = \prod_{1 \leq k < n/2}
  (x-\sigma(\alpha_{k}z)) = \sigma u.
  $$
  Since this holds for all $\sigma$, we have $u\in F[x]$.

  \ref{Tsqfree-2} follows from the recursion \ref{Trecursion}, and
  \ref{Tsqfree-3} from \cite{lidmul93a}, Lemma
  2.6(iii). \ref{Tsqfree-3b} follows from \ref{eq:Trecursion} and
  \ref{Tsqfree-3}. The claim in \ref{Tsqfree-4} is Lemma 2.6(ii) of
  \cite{lidmul93a}. It also follows inductively from \ref{Trecursion},
  as does \ref{Tsqfree-5}.
\end{proof}

 The following normal form for the
decompositions in Ritt's Second Theorem yields an exact count
of equal-degree collisions (\ref{lem:div-b}).
\begin{theorem}\label{th:fifi}
  Let $F$ be a field of characteristic $p \geq 0$, let $m>l\geq 2$ be
  integers, and $n=lm$.  For any $f, g, h,
  g^{*}, h^{*} \in F[x]$ satisfying \ref{eq:RST-1} through
  \ref{eq:intersec}, either \short\ref{th:fifi-1} or
  \short\ref{th:fifi-2} hold. %, and \short\ref{Tsqfree-2} is also valid.
  \begin{enumerate}
  \item\label{th:fifi-1} (First Case) There exists a monic
    polynomial $w \in F[x]$ of degree $s$ and $a \in F$ so that
    \begin{equation}\label{eq:mopo}
      f= (x^{k\ell}w^{\ell}(x^{\ell}))^{[a]}= (x-\parTwo^{kl}w^{l}(\parTwo^{l})) \circ x^{kl}w^{l}(x^{l})
      \circ (x+\parTwo),
    \end{equation}
   where $m=sl+k$ is the division with remainder of $m$ by $l$, with
   $1 \leq k < l$. Furthermore
    \begin{align}
      \label{eq:unidet-1} (g,h)& = (x^{k}w^{\ell},x^{\ell})^{[a]}\\
      & = \big((x-\parTwo^{kl}w^{l}(\parTwo^{l})) \circ x^{k}w^{l} \circ
      (x+\parTwo^{l}),(x-\parTwo^{l}) \circ x^{l} \circ
      (x+\parTwo)\big), \nonumber  \\
      \label{eq:unidet-3} (g^{*}, h ^{*}) & = (x^{\ell}, x^{k} w
      (x^{\ell}))^{[a]} = \big((x- \parTwo^{kl}w^{l}(\parTwo^{l})) \\
      & \quad \circ x^{l} \circ (x+\parTwo^{k}w(\parTwo^{l})),
      (x-\parTwo^{k}w(\parTwo^{l})) \circ x^{k}w(x^{l}) \circ
      (x+\parTwo)\big), \nonumber \\
      \label{eq:unidet} kw+lxw' &\neq 0 \text{ and } p \nmid l.
    \end{align}
 If $p\nmid m$, then
    $(w,\parTwo)$ is uniquely determined by $f$ and $l$.
  \item\label{th:fifi-2} (Second Case) There exist $z,\parTwo \in F$
    with $z \neq 0$ so that
    \begin{align}\label{eq:TN}
      f = T_{n}(x,z)^{[a]} = (x-T_{n}(\parTwo,z)) \circ T_{n}(x,z)
      \circ (x+\parTwo).
    \end{align}
    Now $(z,\parTwo)$ is uniquely determined by $f$ and $l$. Furthermore we have
    \begin{align}
      \label{eq:ab} (g,h) & =   (T_{m}(x,z^{\ell}) ,
      T_{\ell}(x,z))^{[a]} =
      \big((x-T_{n}(\parTwo,z)) \circ T_{m}(x,z^{l})\\ & \quad \circ (x +
      T_{l}(\parTwo,z)), (x-T_{l}(\parTwo,z)) \circ T_{l}(x,z)\circ
      (x+\parTwo)\big), \nonumber \\
      \label{eq:ab-3} (g^{*}, h^{*})& = (T_{\ell}(x,z^{m}),
      T_{m}(x,z))^{[a]}= \big((x-T_{n}(\parTwo,z)) \circ T_{l}(x,z^{m}) \\ &
     \quad  \circ (x+T_{m}(\parTwo,z)), (x-T_{m}(\parTwo,z)) \circ
      T_{m}(x,z) \circ (x+\parTwo)\big), \nonumber \\
     \label{eq:ab-5}  & p \nmid \dg.
    \end{align}

      \item\label{th:fifi-4}
    Conversely, any $(w,\parTwo)$ as in \short\ref{th:fifi-1} for which \ref{eq:unidet}
    holds, and any $(z,\parTwo)$ as in \short\ref{th:fifi-2} when \ref{eq:ab-5} holds,
    yields a collision satisfying \ref{eq:RST-1} through
    \ref{eq:intersec}, via the above formulas.

  \item\label{th:fifi-3} When $l \geq 3$, the First and Second Cases
    are mutually exclusive. For $l=2$, the Second Case is included in
    the First Case.
  \end{enumerate}
\end{theorem}
%%% Then the following hold.
%%% \begin{enumerate}
%%% \item\label{lem:div-1/2}
%%%   $$
%%%   t = \frac{1}{2} (2q^{s+3}+ (1+\delta_{p,2}-\delta_{l,2})
%%%   (q^{4}-q^{3}))(1-q^{-1}).
%%%   $$
%%% \item\label{lem:div-1}

%%%   $$
%%%   q^{s+3}(1-q^{-1})\leq t \leq (q^{s+3}+ q^{4}-q^{3})(1-q^{-1}).
%%%   $$
%%% \item\label{lem:div-1/3} If $n/l\geq 2l$, then
%%%   $$
%%%   t \leq q^{n/l^2+3}.
%%%   $$
%%% \end{enumerate}
%%% \end{theorem}
\begin{proof}
  %%% A polynomial $f$ in the intersection can be presented as
  %%% \begin{equation}\label{eq:intersec}
  %%%   f = g \circ h = g^{*} \circ h^{*}
  %%% \end{equation}
  %%% with $(g,h)$ and $(g^{*}, h^{*})$ normal, $g$, $h^{*} \in
  %%% P^{=}_{\dg/l}$ and $h,g^{*} \in P^{=}_{l}$.

  %%% \short\ref{lem:div-1}
  %%% The only way in which we use the current assumption
  %%% that $p \nmid \dg$ is by noting that $g' (g^{*})' \neq 0$, so
  %%% that
  %%% \ref{eq:RST-2} and the other assumptions of \ref{RST} are
  %%% satisfied
  %%% with $m=\dg/l> l \geq 2$, and we treat the two Cases separately.
   By assumption, either the First or the Second
  Case of Ritt's Second Theorem (\ref{RST}) applies.

  \short\ref{th:fifi-1} From the First Case in \ref{RST}, we have a
  positive integer $K$, linear polynomials $v_{1}$, $v_{2}$, $v_{3}$,
  $v_{4}$ and a nonzero polynomial $W$ with $d = \deg W =(m-K)/l $ and
  (renaming $v_{2}$ as $v_{2}^{-1}$)
  \begin{align*}
    x^{K}W^{l} &=  v_{1} \circ g \circ v_{3},\\
    x^{l} &= v_{3}^{-1} \circ h \circ v_{2}^{-1},\\
    x^{l} &= v_{1} \circ g^{*} \circ v_{4},\\
    x^{K}W(x^{l}) &= v_{4}^{-1} \circ h^{*} \circ v_{2}^{-1}.
  \end{align*}
  We abbreviate $r=~lc(W)$, so that $r\neq 0$, and write $v_{i} =
  a_{i} x+b_{i}$ for $1 \leq i \leq 4$ with all $a_{i}$, $b_{i}\in F$
  and $a_{i} \neq 0$, and first express $v_{3} $, $v_{4}$, and $v_{1}$
  in terms of $v_{2}$. We have
  \begin{align*}
    h &= v_{3} \circ x^{l} \circ v_{2} = a_{3}(a_{2}x+b_{2})^{l}+b_{3},\\
    h^{*} &= v_{4} \circ x^{K}W(x^{l}) \circ v_{2} =
    a_{4}(a_{2}x+b_{2})^{K} \cdot W((a_{2}x+b_{2})^{l})+b_{4}.
  \end{align*}
  Since $h$ and $h^{*}$ are monic original and $K+ld=m$, it
  follows that
  $$
  a_{3}= a_{2}^{-l}, \ b_{3}=-a_{2}^{-l}b_{2}^{l}, \ a_{4}=
  a_{2}^{-m}r^{-1}, \ b_{4}= -a_{2}^{-m}b_{2}^{K}r^{-1}W(b_{2}^{l}).
  $$
  Playing the same game with $g$, we find
  \begin{align*}
    g=v_{1}^{-1} \circ x^{K} W^{l} \circ v_{3}^{-1}&= a_{1}^{-1}
    \bigl( (\frac{x-b_{3}}{a_{3}})^{K} W^{l} (\frac{x-b_{3}}{a_{3}})
    -b_{1}
    \bigr),\\
    a_{1} &= a_{2}^{n} r^{l}, \\
    b_{1} &= b_{2}^{Kl}W^{l}(b_{2}^{l}).
  \end{align*}
  Then
  $$
  g^{*} = v_{1}^{-1} \circ x^{l} \circ v_{4}^{-1} = a_{1}^{-1} \bigl(
  (\frac{x-b_{4}}{a_{4}})^{l}-b_{1} \bigr)
  $$
  is automatically monic original. Furthermore, we have $d =
  (m-K)/l \leq \lfloor m/l \rfloor = s$ and
  \begin{equation}
    \label{adm}
    f = v_{1}^{-1} \circ (v_{1} \circ g \circ
    v_{3}) \circ (v_{3}^{-1} \circ h \circ v_{2}^{-1}) \circ v_{2}= v_{1}^{-1}
    \circ x^{Kl}\cdot W^{l}(x^{l}) \circ v_{2}.
  \end{equation}
  %%% We have the $d+1$ coefficients $w_{i} \in \mathbb{F}_{q}$
  We set
  \begin{align*}\parTwo&=\frac{b_{2}}{a_{2}}\in F, \quad
    u_{1}=x+\frac{b_{1}}{a_{1}}=\frac{v_{1}}{a_{1}}, \quad
    u_{2}=x+\parTwo=\frac{v_{2}}{a_{2}}, \quad\\
    %%% plus the four $a_{1}$, $b_{1}$,
    %%% $a_{2}$, $b_{2}$ of $v_{1}$ and $v_{2}$, with $a_{1}a_{2}w_{d}
    %%% \neq 0$.
    %%% The following substitution reduces the total of $d+5$
    %%% coefficients by
    %%% $2$. We set
    %%% V_{1}= \frac{1}{w_{d}^{l}a_{2}^{\dg}}\cdot v_{1}, \quad
    %%% V_{2}=\frac{1}{a_{2}}\cdot v_{2}, \quad
    %%% w = \frac{x^{s-d}W(a_{2}^{l}x)}{a_{2}^{dl}~lc (W)}.
    w&=r^{-1}a^{-ld}_{2}x^{s-d}\cdot W(a^{l}_{2}x) \in
    F[x].
  \end{align*}
  Then $b_{1}/a_{1}= a^{kl}w^{l}(a^{l})$, $w$ is monic of degree $s$,
  $ u_{1}^{-1}= x-b_{1}/a_{1}= x - \parTwo^{kl}
  w^{l}(\parTwo^{l})$, and
\begin{align}\label{eq:mondegs}
W(x)=ra^{ls}_{2}x^{-(s-d)}w(a_{2}^{-l}x).
\end{align}

   Noting that $m = ld + K =ls + k$, the equation
  analogous to \ref{adm} reads
  \begin{align}
    (x^{k\ell} w^{\ell}(x^{\ell}))^{[a]}& = u_{1}^{-1} \circ x^{kl}
    w^{l}(x^{l})\circ u_{2} \nonumber \\
    &= a_{1}\cdot v^{-1}_{1} \circ x^{kl} \cdot
    \frac{x^{l^{2}(s-d)}W^{l}(a_{2}^{l}x^{l})}{a_{2}^{dl^{2}}r^{l}}
    \circ
    \frac{v_{2}}{a_{2}} \nonumber \\
    &= v_{1}^{-1}\circ a_{2}^{\dg}r^{l}\cdot
    \bigl(\frac{v_{2}}{a_{2}}\bigr)^{kl} \cdot
    \bigl(\frac{v_{2}}{a_{2}}\bigr)^{l^{2}(s-d)} \cdot
    \frac{W^{l}(v_{2}^{l})}{a_{2}^{dl^{2}}r^{l}} \nonumber \\
    &= v_{1}^{-1}\circ x^{Kl} \cdot W^{l}(x^{l})\circ v_{2} = f.
    \label{compFirst}
 \end{align}
  This proves the existence of $w$ and $\parTwo$, as claimed in
  \ref{eq:mopo}.

  In order to express the four components in the new parameters, we
  note that $K=k+l(s-d)$. Thus
  \begin{align*}
    g &= v_{1}^{-1} \circ x^{K}W^{l} \circ v_{3}^{-1} \\
    &= (r^{-l}a_{2}^{-n} x-\parTwo^{kl}w^{l}(\parTwo^{l})) \circ
    (a_{2}^{l}(x+\parTwo^{l}))^{K} \cdot
    W^{l}(a_{2}^{l}(x+\parTwo^{l})) \\
    &= r^{-l}a_{2}^{-n} \bigl(a_{2}^{Kl}(x+\parTwo^{l})^{K} \cdot
    r^{l}a_{2}^{l^{2}s} a_{2}^{-l^{2}(s-d)}
    (x+\parTwo^{l})^{-l(s-d)} w^{l}(x+\parTwo^{l})\bigr)\\
    & \quad-\parTwo^{kl}w^{l}(\parTwo^{l}) \\
    &= a_{2}^{-n+Kl+l^{2}s-l^{2}s+l^{2}d} (x+\parTwo^{l})^{K-ls+ld}
    w^{l}(x+\parTwo^{l})
    -\parTwo^{kl}w^{l}(\parTwo^{l}) \\
    &= (x+\parTwo^{l})^{k} w^{l}(x+\parTwo^{l}) -\parTwo^{kl}
    w^{l}(\parTwo^{l})\\
    &= \bigl(x-\parTwo^{kl} w^{l}(\parTwo^{l})\bigr) \circ x^{k}w^{l} \circ
    (x+\parTwo^{l}) = (x^{k}w^{\ell})^{[a^{\ell}]}, \\ %Ende von 4!
    h &= v_{3} \circ x^{l} \circ v_{2}
    = a_{2}^{-l}(a_{2}x+b_{2})^{l} - a_{2}^{-l}b_{2}^{l}\\
    &= (x-\parTwo^{l}) \circ x^{l} \circ (x+\parTwo) = (x^{\ell})^{[a]}, \\
    g^{*} &= v_{1}^{-1} \circ x^{l} \circ v_{4}^{-1} \\
    &= (r^{-l}a_{2}^{-n} x-\parTwo^{kl} w^{l}(\parTwo^{l})) \circ
    \bigl(ra_{2}^{m}(x+r^{-1}a_{2}^{-m} b_{2}^{K} \cdot
    W(b_{2}^{l}))\bigr)^{l} \\
    &= (x+r^{-1} a_{2}^{-m} b_{2}^{K} \cdot r a_{2}^{ls}
    b_{2}^{-l(s-d)} w(\parTwo^{l}))^{l} -
    \parTwo^{kl}w^{l}(\parTwo^{l}) \\
    &= \bigl(x+a_{2}^{-k}b_{2}^{k}w(\parTwo^{l})\bigr)^{l}
    -\parTwo^{kl}w^{l}(\parTwo^{l}) \\
    &= (x-\parTwo^{kl}w^{l}(\parTwo^{l})) \circ x^{l} \circ
    (x+\parTwo^{k}w(\parTwo^{l}))= (x^{\ell})^{[a^{k}w(a^{\ell})]},
\end{align*}
\begin{align*}
    h^{*} &= v_{4} \circ x^{K} W(x^{l}) \circ v_{2} \\
    &=\bigl(r^{-1}a_{2}^{-m} (x-b_{2}^{K} W(b_{2}^{l}))\bigr) \circ
    (a_{2}(x+\parTwo))^{K} W(a_{2}^{l}(x+\parTwo)^{l}) \\
    &= r^{-1}a_{2}^{-m} \cdot ra_{2}^{ls} \cdot \bigl((a_{2}^{K}
    (x+\parTwo)^{K} (a_{2}^{l}(x+\parTwo)^{l}))^{-(s-d)}
    w((x+\parTwo)^{l})\\
    &\quad  - b_{2}^{K} b_{2}^{-l(s-d)} w(\parTwo^{l})\bigr) \\
    &= a_{2}^{-k} \bigl(a_{2}^{K-l(s-d)} (x+\parTwo)^{K-l(s-d)}
    w((x+\parTwo)^{l})
    -b_{2}^{K-l(s-d)} w(\parTwo^{l}) \bigr)\\
    &= (x+\parTwo)^{k} w((x+\parTwo)^{l})- \parTwo^{k}w(\parTwo^{l})\\
    &= (x-\parTwo^{k}w(\parTwo^{l})) \circ x^{k}w(x^{l}) \circ
    (x+\parTwo)  = (x^{k}w(x^{\ell}))^{[\parTwo]}.
  \end{align*}
  In the right hand component $x+\parTwo$, the constant
  $\parTwo$ is arbitrary. All other linear components follow
  automatically from the required form of $g$, $h$, $g^{*}$, $h^{*}$,
  namely, being monic original, and from the condition that $g$
  and $h$ (and $g^{*}$ and $h^{*}$) have to match up with their
  ``middle'' components.  Furthermore, we have
  \begin{align}\label{eq:kw}
\begin{aligned}
         0  = g'=(x^{k-1}w^{l-1}(kw+lxw')) \circ (x+\parTwo^{l})
      &\Longleftrightarrow
      kw+ lxw'=0,\\
      0  = (g^{*})' = lx^{l-1} \circ (x+\parTwo^{k}w(\parTwo^{l}))
     &\Longleftrightarrow p \mid l.
\end{aligned}
    \end{align}

  Thus \ref{eq:unidet} follows from \ref{eq:RST-2}.

  In order to prove the uniqueness if $p\nmid m$, we take monic $w$,
  $ \tilde w \in F[x]$ of degree $s$, and $\parTwo$, $\tilde \parTwo \in F$ and
  the unique monic linear polynomials $v$ and $\tilde v$ for which
  \begin{align}\label{eq:unique}
    f = v \circ x^{kl}w^{l}(x^{l}) \circ (x+\parTwo)= \tilde v \circ x^{kl}
    \tilde{w}^{l}(x^{l}) \circ (x+\tilde \parTwo).
  \end{align}

  By composing on the left and right with $\tilde{v}^{-1}$ and
  $(x+\tilde \parTwo)^{-1}$, respectively, and abbreviating
  $u=\tilde{v}^{-1}\circ v$, we find
  \begin{align*}
    x^{kl}\tilde{w}^{l}(x^{l})&=\tilde{v}^{-1}\circ v\circ
    x^{kl}w^{l}(x^{l})\circ(x+\parTwo)\circ(x-\tilde \parTwo)\\
    &=u\circ x^{kl}w^{l}(x^{l})\circ(x+\parTwo-\tilde \parTwo).
  \end{align*}
  Since $l\geq 2$ and the left hand side is a polynomial in $x^{l}$,
  its second highest coefficient (of $x^{n-1}$) vanishes. Equating
  this with the same coefficient on the right, and abbreviating
  $a^{*}=\parTwo-\tilde \parTwo$, we find
  $$
  0=kla^{*}+sl^{2}a^{*}=na^{*},
  $$
  so that $a^{*}=0$, since $p \nmid \ell$ by \ref{eq:unidet} and hence
  $p\nmid\dg$. Thus $\parTwo=\tilde \parTwo$ and
  $$
  x^{k}\tilde{w}^{l}\circ x^{l}=x^{kl}\tilde{w}^{l}(x^{l})=u\circ
  x^{kl}w^{l}(x^{l})=u\circ x^{k}w^{l}\circ x^{l},
  $$
  %Ende 1 18.Mai 2009
 $$x^k \tilde w^l =u \circ x^k w^l.$$
  Now $x^{k}\tilde w^{l}$ and $x^{k}w^{l}$ are monic original, since $k\geq 1$.
  %%% so that $v$ and $\tilde v$ are determined
  %%% by the leading and trailing coefficients of
  It follows that $u=x$ and $w^l= \tilde w^l$. Both polynomials are
  monic, so that $w=\tilde w$, as claimed. (The equation
  \ref{eq:unidet-1} for $h$ determines $a$ uniquely provided that $p
  \nmid l$, even if $p \mid m$. However, the value of $h$ may not be
  unique in the latter case.)

  \short\ref{th:fifi-2} In the Second Case, again renaming $v_{2}$ as
  $v_{2}^{-1}$, and also $z$ as $z_{2}$, we have from \ref{RST}
  \begin{align*}
    T_{m}(x,z_{2}^{l}) &= v_{1} \circ g \circ v_{3},\\
    T_{l}(x,z_{2})&= v_{3}^{-1} \circ h \circ v_{2}^{-1},\\
    T_{l}(x,z_{2}^{m})&= v_{1} \circ g^{*} \circ v_{4},\\
    T_{m}(x,z_{2})&= v_{4}^{-1} \circ h^{*} \circ v_{2}^{-1},\\
    h &= v_{3} \circ T_{l}(x,z_{2}) \circ v_{2} =
    a_{3}T_{l}(a_{2}x+b_{2},z_{2})+b_{3},\\
    h^{*} &= v_{4} \circ T_{m}(x,z_{2})\circ v_{2}=
    a_{4}T_{m}(a_{2}x+b_{2},z_{2})+b_{4}.
  \end{align*}
  As before, it follows that
  $$
  a_{3}=a_{2}^{-l}, \quad b_{3}=-a_{2}^{-l}T_{l}(b_{2},z_{2}), \quad
 a_{4}=a_{2}^{-m}, \quad b_{4}=-a_{2}^{-m}T_{m}(b_{2},z_{2}).
  $$
  Furthermore, we have
  \begin{align*}
    g &= v_{1}^{-1} \circ T_{m}(x,z_{2}^{l}) \circ v_{3}^{-1} =
    a_{1}^{-1}(T_{m}(a_{3}^{-1}(x-b_{3}),z_{2}^{l})-b_{1}),\\
    a_{1}&= a_{2}^{\dg},\\
    b_{1} &= T_{m}(T_{l}(b_{2}, z_{2}),z_{2}^{l})=T_{\dg}(b_{2}, z_{2}),\\
    f &= \bigl(a_{2}^{-\dg}(x-T_{\dg}(b_{2},z_{2}))\bigr) \circ
    T_{\dg}(x,z_{2}) \circ (a_{2}x+b_{2}).
  \end{align*}

  We now set $\parTwo=b_{2}/a_{2}$ and $z=z_{2}/a_{2}^{2}$ and show that the
  preceding equation holds with $(1,\parTwo,z)$ for $(a_{2}, b_{2},
  z_{2})$. \ref{Tsqfree-4} with $t=a_{2}^{-1}$ says that
\begin{align*}
a_{2}^{-\dg}T_{\dg}(a_{2}x+b_{2}, z_{2})&=T_{\dg}(x+\parTwo,z),\\
a_{2}^{-\dg}T_{\dg}(b_{2},z_{2})&= T_{\dg}(\parTwo,z),\\
f&=(x-T_{\dg}(\parTwo,z))\circ T_{\dg}(x,z)\circ(x+\parTwo).
\end{align*}
Thus the first claim in \short\ref{th:fifi-2} holds with these
values.
In the same vein, applying \ref{Tsqfree-4} with $t$ equal to
$a_{2}^{-1}, a_{2}^{-l}, a_{2}^{-m}, a_{2}^{-1}$, respectively, yields
\begin{align*}
a_{2}^{-l}T_{l}(a_{2}x+b_{2},z_{2})&=T_{l}(x+\parTwo,z),\\
a_{2}^{-\dg}T_{m}(a_{2}^{l}x+T_{l}(b_{2},z_{2}),z_{2}^{l})&=T_{m}(x+a_{2}^{-l}T_{l}(b_{2},z_{2}),z^{l})\\
&=T_{m}(x+T_{l}(\parTwo,z),z^{l}),\\
a_{2}^{-\dg}T_{l}(a_{2}^{m}x+T_{m}(b_{2},z_{2}),z_{2}^{m})&=T_{l}(x+a_{2}^{-m}T_{m}(b_{2},z_{2}),z^{m})\\
&=T_{l}(x+T_{m}(\parTwo,z),z^{m}),\\
a_{2}^{-m}T_{m}(a_{2}x+b_{2}, z_{2})&=T_{m}(x+\parTwo,z).
\end{align*}

For the four components, we have
  \begin{align*}
    g  &= v_{1}^{-1} \circ T_{m}(x,z_{2}^{l}) \circ v_{3}^{-1}\\
    &= a_{2}^{-\dg}(x-T_{\dg}(b_{2},z_{2})) \circ T_{m}(x,z_{2}^{l})
    \circ (a_{2}^{l}x+T_{l}(b_{2},z_{2}))\\
    &=a_{2}^{-\dg}T_{m}(a_{2}^{l}x+T_{l}(b_{2},z_{2}),z_{2}^{l})
    -a_{2}^{-\dg}T_{m}(T_{l}(b_{2},z_{2}),z_{2}^{l})\\
    &=T_{m}(x+T_{l}(\parTwo,z),z^{l})-T_{\dg}(\parTwo,z)\\
    &=(x-T_{\dg}(\parTwo,z))\circ
    T_{m}(x,z^{l})\circ(x+T_{l}(\parTwo,z))= T_{m}(x,z^{\ell})^{[T_{\ell}(a,z)]},\\
    h &= v_{3} \circ T_{l}(x,z_{2}) \circ v_{2} = a_{2}^{-l}
    T_{l}(a_{2}    x+b_{2},z_{2})-a_{2}^{-l} T_{l}(b_{2},z_{2})\\
    &= a_{2}^{-l}(x-T_{l}(b_{2},z_{2})) \circ T_{l}(x,z_{2}) \circ (a_{2}x+b_{2})\\
    &=T_{l}(x+\parTwo,z)-T_{l}(\parTwo,z)\\
    &=(x-T_{l}(\parTwo,z))\circ T_{l}(x,z)\circ (x+\parTwo) =
    T_{\ell}(x,z)^{[a]},\\
    g^{*}  &= v_{1}^{-1} \circ T_{l} (x,z_{2}^{m}) \circ v_{4}^{-1}\\
    &= a_{2}^{-\dg}(x-T_{\dg}(b_{2},z_{2})) \circ T_{l}(x,z_{2}^{m})
    \circ (a_{2}^{m}x+T_{m}(b_{2},z_{2}))\\
    &=a_{2}^{-\dg}T_{l}(a_{2}^{m}x+T_{m}(b_{2},z_{2}),z_{2}^{m})-a_{2}^{-\dg}T_{\dg}
    (b_{2},z_{2})\\
    &= T_{l}(x+T_{m}(\parTwo,z),z^{m})-T_{\dg}(\parTwo,z)\\
    &=(x-T_{\dg}(\parTwo,z))\circ T_{l}(\parTwo,z^{m})\circ
    (x+T_{m}(\parTwo,z)) = T_{\ell}(a,z^{m})^{[T_{m}(a,z)]},\\
    h^{*}  &= v_{4} \circ T_{m}(x,z_{2}) \circ v_{2}\\
    &= a_{2}^{-m}(x-T_{m}(b_{2},z_{2})) \circ T_{m}(x,z_{2})\circ
    (a_{2}x+b_{2})\\
    &=a_{2}^{-m}T_{m}(a_{2}x+b_{2},z_{2})-a_{2}^{-m}T_{m}(b_{2},z_{2})\\
    &=T_{m}(x+\parTwo,z)-T_{m}(\parTwo,z)\\
    &=(x-T_{m}(\parTwo,z))\circ T_{m}(x,z)\circ (x+\parTwo) = T_{m}(x,z)^{[a]}.
  \end{align*}

  Since
  $$
  0 \neq g' = T_{m}'(x,z^{l})\circ (x+T_{l}(\parTwo,z)),
  $$
  \ref{Tsqfree-3b} implies that $p \nmid m$. Similarly, the
  non-vanishing of $(g^{*})'$ implies that $p \nmid l$, and
  \ref{eq:ab-5} follows.

  Next we claim that the representation of $f$ is unique.
%16 einf�gen
So we take some $(z,\parTwo),(z^{*},\parTwo^{*})\in F^{2}$ with
$zz^{*}\neq 0$ and
\begin{align}\label{eq:klam}
(x-T_{\dg}(\parTwo,z))\circ T_{\dg}(x,z)\circ
(x+a)=(x-T_{\dg}(\parTwo^{*},z^{*}))\circ T_{\dg}(x,z^{*})\circ(x+\parTwo^{*}).
\end{align}

  Comparing the coefficients of $x^{\dg-1}$ in \ref{eq:klam} and using
  \ref{Tsqfree-2} yields $\dg \parTwo=\dg \parTwo^{*}$, hence $\parTwo=\parTwo^{*}$, since
  $p\nmid \dg$. We now compose \ref{eq:klam} with $x-\parTwo$ on the right
  and find
$$
(x-T_{\dg}(\parTwo,z))\circ T_{\dg}(x,z)=(x-T_{\dg}(\parTwo,z^{*}))\circ T_{\dg}(x,z^{*}).
$$
  Now the coefficients of $x^{\dg-2}$ yield $-\dg z=-\dg z^{*}$, so
  that $z=z^{*}$.

\short\ref{th:fifi-4}
 For the two converses, we first take some $(w,\parTwo)$ satisfying \ref{eq:unidet}
  and define $f$, $g$, $h$, $g^{*}$ and $h^{*}$ via \ref{eq:mopo} through
  \ref{eq:unidet-3}. Then \ref{eq:RST-1}, \ref{eq:RST-3}, and
  \ref{eq:intersec} hold. As to \ref{eq:RST-2}, we have $p \nmid l $
  from \ref{eq:unidet}, and hence $(g^{*})' \neq 0$. Furthermore,
  $$
  (x^{k}w^{l})'= x^{k-1}w^{l-1}\cdot (kw+lxw')\neq 0,
  $$
  so that also $g' \neq 0$.

 Furthermore, any $(z,\parTwo)$ with $z\neq 0$ and
  \ref{eq:ab-5} yields a
  collision as prescribed, since \ref{eq:ab-5} and
  \ref{Tsqfree-3b} imply that $T_{m}'(x,z^{l}) T_{l}'(x,z^{m}) \neq 0$.

  \short\ref{th:fifi-3} We first assume $l \geq 3$ and show that the
  First and Second Cases are mutually exclusive. Assume, to the
  contrary, that in our usual notation we have
  \begin{equation}
    \label{exptrig1}
    f = v_1 \circ x^{kl}w^{l}(x^{l}) \circ (x+\parTwo)=
    v_2 \circ T_n(x,z) \circ (x+\parTwo^{*}),
  \end{equation}
  where $v_{1}$ and $v_{2}$ are the unique linear polynomials that
  make the composition monic original, as specified in
  \short\ref{th:fifi-1} and \short\ref{th:fifi-2}. Then
  \begin{align*}
    f &= (v_1 \circ x^{k}w^{l}\circ (x+\parTwo^l)) \circ ((x+\parTwo)^l-\parTwo^l) \\
    &= \bigl(v_2 \circ T_{m} (x+T_l(\parTwo^{*},z),z^{l} )\bigr) \circ
    (T_l(x+\parTwo^{*},z)-T_l(\parTwo^{*},z)).
  \end{align*}

  These are two monic original decompositions of $f$, and since $p \nmid m$ by
  \ref{eq:ab-5}, %doppelt
  the uniqueness of tame decompositions implies that
  \begin{align}
    \label{exptrig3}
    h &= (x+\parTwo)^l-\parTwo^l =T_l(x+\parTwo^{*},z)-T_l(\parTwo^{*},z),\\
    h' &= l(x+\parTwo)^{l-1} =T'_l(x+\parTwo^{*},z).  \nonumber
  \end{align}

  If $p = 0$ or $p \geq 3$, then according to \ref{Tsqfree-1}, $T'_l(x,z)$
  is squarefree, while $(x+\parTwo)^{l-1}$ is not, since $l\geq 3$. This
  contradiction refutes the assumption \ref{exptrig1}.

  If $p=2$, then $l$ is odd by \ref{eq:ab-5}. After adjoining a square
  root $z_{0}$ of $z$ to $F$ (if necessary), \ref{Tsqfree-1b} implies that
  $T'_{l}(x,z)=((x-2z_{0})u^{2}+2z_{0}^{\dg})'=u^{2}$ has $(l-1)/2$ distinct roots in an
  algebraic closure of $F$, while $(x+\parTwo)^{l-1}$ has only one. This
  contradiction is sufficient for $l \geq 5$. For $l=3$, we have
  $T_{3}= x^{3}-3 yx$ and there are no $a$, $\parTwo^{*}$, $z\in F$ with $z \neq 0$
  so that
  \begin{align*}
    x^{3} + ax^{2} + a^{2}x  &= (x+a)^{3}-a^{3}= (x+\parTwo^{*})^{3} -3z
    (x+\parTwo^{*})-((\parTwo^{*})^{3} -3z\parTwo^{*})\\
    &= x^{3} + a^{*} x^{2} +((\parTwo^{*})^{2}+z)x.
  \end{align*}
  Again, \ref{exptrig1} is refuted.

  For $l=2$, we claim that any composition
  $$
  f = (x-T_{\dg}(a,z)) \circ T_{m}(x,z^{2}) \circ T_{2}(x,z) \circ (x+a)
  $$
  of the Second Case already occurs in the First Case. We have
  $2m=\dg$ and $T_{2}=x^{2}-2y$. Since $m$ is odd by \ref{eq:RST-1}
  and $p \nmid m$ by \ref{eq:ab-5}, \ref{Tsqfree-1b} guarantees a monic
  $u\in F[x]$ of degree $s = (m-1)/2$ with $T_{m}(x,z^{2})
  =T_{m}(x,(-z)^{2})= (x+2z)u^{2}-2z^{m}$. Then for $w = u
  \circ (x-2z)$ we have
  \begin{align*}
    f & = (x-T_{\dg}(a,z)) \circ ((x+2z)u^{2}-2z^{m}) \circ (x^{2}-2z)
    \circ (x+a)\\
    &= (x-T_{\dg}(a,z)  -2z^{m}) \circ (x^{2} \cdot w^{2} (x^{2}))
    \circ (x+a),
  \end{align*}
  which is of the form \ref{eq:mopo}, with $k = m-2s = 1$.
\end{proof}

The quantity in \ref{eq:unidet} is the logarithmic derivative of
$x^{k} w^{\ell}$, multiplied by $xw$.

\begin{remark}
Given just $f \in F[x]$, how can we determine whether Ritt's Second
Theorem applies to it, and if so, compute $(w,a)$ or $(z,a)$, as
appropriate? We may assume $f$ to be monic and original of degree
$n$. The divisor $l$ of $n$ might be given as a further input, or we
perform the following for all divisors $l$ of $n$ with $2 \leq l \leq
\sqrt{n}$ and $\gcd(l,n/l)=1$. If $p \nmid n$, the task is easy. We
compute, by standard methods (\cite{kozlan89}; \cite{gat90c}) decompositions
$$
f = g \circ h = g^{*} \circ h^{*}
$$
with $~deg h = ~deg g^{*}=l$ and all components monic original. If
one of these decompositions does not exist, Ritt's Second Theorem does
not apply; otherwise the components are uniquely determined. If
$h_{l-1}$ is the coefficient of $x^{l-1}$ in $h$, then $a =
h_{l-1}/l$ in \ref{eq:mopo}. Furthermore,
\begin{align*}
g(-a^{l}) &= -a^{kl}w^{l}(a^{l}),\\
g \circ (x-a^{l})-g(-a^{l})&= x^{k}w^{l},
\end{align*}
from which $w$ is easily determined via an $x$-adic Newton iteration
for extracting an $l$th root of the reversal of the left hand side,
divided by $x^{k}$.

If the Second Case applies, then by \ref{Tsqfree-2} the three highest
coefficients in $f$ are
\begin{align*}
f  &= x^{n}+f_{n-1}x^{n-1}+f_{n-2}x^{n-2}+O(x^{n-3})\\
 &= (x+a)^{n}-nz(x+a)^{n-2}+O(x^{n-4})\\
&= x^{n}+na x^{n-1}+\bigl(\frac{n(n-1)}{2}a^{2}-nz \bigr) x^{n-2}+O(x^{n-3});
\end{align*}
this determines $a$ and $z$.

The case where $f'=0$, so that $p \mid \dg$, is reduced to $f' \neq 0$
by \ref{fact} below. Finally, we are left with the situation where $f'
\neq 0$ and $p \mid \dg$. By \ref{eq:unidet} and \ref{eq:ab-5}, we are
then in the First Case with $p \mid m$. This scenario is still not
completely understood (see \citealt{gat10}), and I am not aware of an
efficient method for calculating the decompositions, if any.
\end{remark}

\begin{remark}\label{rem:opara}
  Other parametrizations are possible. As an example, in the Second
  Case, for odd $q=p$, one can choose a nonsquare
  $z_{0}\in F = \mathbb{F}_{q}$ and $B=\{1,\dots,(q-1)/2 \}$. Then all $f$
  in \ref{eq:TN} can also be written as
$$
f=b^{-\dg}(x-T_{\dg}(a,z))\circ T_{\dg}(x,z)\circ(bx+a)
$$
with unique $(z, a, b) \in \{1, z_{0}\} \times F \times B=Z$. To wit,
let $z,a \in F$ with $z\neq 0$. Take the unique $(z^{*}, a^{*}, b) \in
Z$, so that $z^{*}=b^{2}z$ and $a^{*}=ab$. Then $z^{*}$ is determined
by the quadratic character of $z$, and $b$ by the fact that every
square in $F^{\times}$ has a unique square root in $B$; the other one
is $-b\in F^{\times}\setminus B$. \ref{Tsqfree-4} says that
$$
b^{\dg}T_{\dg}(x,z)=T_{\dg}(bx,z^{*}),
$$
\begin{align*}
(x-T_{\dg}(a,z))\circ T_{\dg}(x,z)\circ
(x+a)&=b^{-\dg}(x-T_{\dg}(a^{*},z^{*}))\circ T_{\dg}(bx,z^{*})\circ(x+a)\\
&=b^{-\dg}(x-T_{\dg}(a^{*},z^{*}))\circ T_{\dg}(x,z^{*})\circ (bx+a^{*}),
\end{align*}
as claimed.
If $F$ is algebraically closed, as in \cite{zan93}, we can take
$z=1$. The reduction from finite fields to this case is provided by
\cite{sch00c}, Section 1.4, Lemma 2.
\end{remark}

\begin{remark}\label{rem:rid}
  If $p\nmid\dg$, then we can get rid of the right hand component
  $x+\parTwo$ by a further normalization. Namely, when
  $f=x^{\dg}+\sum_{{0\leq i<\dg}}{f_{i}x^{i}},$ then
  $f\circ(x+\parTwo)=x^{\dg}+(\dg \parTwo+f_{\dg-1})x^{\dg-1}+O(x^{\dg-2})$. We
  call $f$ \emph{second-normalized} if $f_{\dg-1}=0$.  (This has been
  used at least since the times of Cardano and Tartaglia.) For any
  $f,$ the composition $f\circ(x-f_{\dg-1}/\dg)$ is second-normalized,
  and if
  \begin{align}\label{eq:secnorm}
    \deg g = m\text{ and } f = g\circ
    h=x^{\dg}+mh_{\dg/m-1}x^{\dg-1}+O(x^{\dg-2})
  \end{align}
  is second-normalized, then so is $h$ (but not necessarily $g$).
\end{remark}

\begin{corollary}\label{cor:andsecond} In \ref{th:fifi}, if
  $p\nmid\dg$ and $f$ is second-normalized, then all claims hold with
  $\parTwo=0$.
\end{corollary}
%% hier 1 einf�gen

For the counting results below, it is convenient to assume $F$ to be
perfect.  Then each element of $F$ has a $p$th root, where $p\geq 2$
is the characteristic. Any finite field is perfect. For any $f \in
F[x]$, we have
\begin{align*}
f \text{ is a Frobenius composition} & \Longleftrightarrow \exists g
\in F[x] \quad f=g \circ x^{p}\\
& \Longleftrightarrow \exists h \in F[x] \quad f = x^{p} \circ h.
\end{align*}
We start by making the first condition in \ref{eq:unidet} more
explicit.

\begin{lemma}\label{lem:root}
  Let $F$ be a perfect field, $l$ and $m$ positive integers
  with $\gcd(l,m)=1$, let $m=ls+k$ and $s=tp+r$ be divisions with remainder,
  so that $1 \leq k < l$ and $0 \leq r < p$,
 and $w \in F[x]$ monic of degree $s$. Then
  \begin{equation}\label{eq:lxw}
   p \nmid l  \text{ and } kw + lxw'= 0  \Longleftrightarrow p \mid m \text{
      and } \exists u \in F[x] \text{ monic} \quad w=x^{r} u^{p}.
  \end{equation}
If the conditions in \ref{eq:lxw} are satisfied, then $u$ is uniquely
determined.
\end{lemma}
\begin{proof}
  For ``$\Longrightarrow$'', we denote by $w^{(i)}$ the $i$th
  derivative of $w$. By induction on $i\geq 0$, we find that
  \begin{align*}
    (k+il) w^{(i)}+ lxw^{(i+1)}  &= 0,\\
    (k+il) w^{(i)}(0) &= 0.
  \end{align*}
  Now $p \nmid s-i$ for $0 \leq i < r$, $p \mid m = k+ls= ~lc
  (kw+lxw')$, and $p \nmid l$.  Thus
  $$
  p \nmid m -(s-i)l=k+ls-ls+il=k+il
  $$
  for $0 \leq i < r$, and hence $w^{(i)}(0)= 0$ for these $i$.  Since
  $r < p$, this implies that the lowest $r$ coefficients of $w$
  vanish, so that $x^{r}\mid w$ and $v = x^{-r}w \in F[x]$. Then
  \begin{align*}
    lv'&= l(-rx^{-r-1}w+x^{-r}w')= x^{-r-1}(-lrw-kw)\\
    &= -x^{-r-1}w\cdot(lr+k)= -x^{-r-1}w \cdot(m-l(s-r))=0.
  \end{align*}
  This implies that $v'=0$ and $v = u^{p}$ for some
  $u\in F[x]$, since $F$ is perfect.

  For ``$\Longleftarrow$'', $p\nmid l$ follows from $\gcd(l,m)=1$, and
  we verify
  \begin{align*}
    kw + lxw' &= kx^{r}u^{p}+ lx \cdot rx^{r-1} u^{p} =
    x^{r} u^{p}(k+lr)\\
    &= w \cdot (m-l(s-r))=0.
  \end{align*}
The uniqueness of $u$ is immediate, since $x^{r}u^{p}= x^{r}
\tilde{u}^{p}$ implies $u = \tilde{u}$.
\end{proof}

We can now estimate the number of non-Frobenius distinct-degree
collisions. If $p\nmid m$, the bound is exact. We use Kronecker's $\delta$ in
the statement.
\begin{theorem}\label{lem:div-b}
  Let $\mathbb{F}_{q}$ be a finite field of characteristic $p$, let
  $l$ and $m$ be integers with $m>l\geq 2$ and $\gcd(l,m)=1$, $n=lm$,
  $s= \lfloor m/l \rfloor$, and $t =\#(D_{\dg, l} \cap D_{\dg,m}  \setminus
  \mathbb{F}_{q}[x^p] )$.  Then the following hold.
  \begin{enumerate}
  \item\label{lem:div-1/2} \label{lem:div-1} If $p \nmid \dg$, then
    $$
    t = q^{s+1}+ (1-\delta_{l,2})
    (q^{2}-q),
    $$
    $$
    q^{s+1}\leq t < q^{s+1}+ q^{2}.
    $$
  \item\label{lem:div-1/4} If $p \mid l$, then $t=0$.
  \item\label{lem:div-1/3} If $p \mid m$, then
    $$
    t \leq q^{s+1}-q^{\lfloor s/p\rfloor +1}.
    $$
  \end{enumerate}
\end{theorem}
\begin{proof}
  \short\ref{lem:div-1/2} The monic original polynomials $f \in
 D_{\dg,l} \cap D_{\dg,m}  \setminus
  \mathbb{F}_{q}[x^p]  =T$ fall either into the
  First or the Second Case of Ritt's Second Theorem. In the First
  Case, such $f$ are injectively parametrized by $(w,\parTwo)$ in
  \ref{th:fifi-1}.  Condition \ref{eq:unidet} is satisfied, since $p
  \nmid m= k+ls =~lc(kw+lxw')$. Thus there are $q^{s+1}$ such pairs.
  In the Second Case, we have the parameters $(z,\parTwo), q^{2}-q$ in
  number, from \ref{th:fifi-2}. Furthermore, \ref{th:fifi-3} says that
  $t$ equals the sum of the two contributions if $l \geq 3$, and it
  equals the first summand for $l=2$; in the latter case, we have
  $p\neq 2$. Both claims in \short\ref{lem:div-1} follow.

  \short\ref{lem:div-1/4} \ref{eq:unidet} and \ref{eq:ab-5} are never
  satisfied, so that $t=0$.

  \short\ref{lem:div-1/3} We have essentially the same situation as in
  \short\ref{lem:div-1/2}, with $p \nmid l$ and $(w,\parTwo)$
  parametrizing our $f$ in the First Case, albeit not
  injectively. Thus we only obtain an upper bound. The first condition
  in \ref{eq:unidet} holds if and only if $w$ is not of the form
  $x^{r} u^{p}$ as in \ref{eq:lxw}. We note that $\deg u= (s-r)/p =
  \lfloor s/p \rfloor$ in \ref{eq:lxw}, so that the number of
  $(w,\parTwo)$ satisfying \ref{eq:unidet} equals $q^{s+1}-q^{\lfloor
    s/p\rfloor+1}$. Since $p \mid m \mid \dg$, \ref{eq:ab-5} does not
  hold, and there is no non-Frobenius decomposition in the Second
  Case.\qed
\end{proof}
%1 verschoben von p.7 auf p.40

This shows \short\ref{lem:div-1}, \short\ref{lem:div-1/4}, and
\short\ref{lem:div-1/3} in \ref{tab:cor-table}.

\begin{example}
  We note two instances of misreading Ritt's Second Theorem.
  \cite{boddeb09} claim in the proof of their Lemma 5.8 that $t \leq
  q^{5}$ in the situation of \ref{lem:div-1}. This contradicts the
  fact that the exponent $s+1$ of $q$ is unbounded.

  A second instance is in \cite{cor90}. The author claims that her
  following example contradicts the Theorem. She takes (in our
  language) positive integers $b$, $c$, $d$, $t$ and elements $h_{0},
  \ldots, h_{t} \in F$ and sets $m = bp^{c}+d$ and $l=p^{c}+1$, where
  $c < p$, $tl \leq m$, and $F$ is a field of characteristic $p
  >0$. Then for
\begin{align*}
h & = \sum_{0 \leq i \leq t}h_{i}x^{m-il},\\
g^{*} & = \sum_{o \leq i, j \leq t} h_{i}h_{j}x^{m-ip^{c}-j}.
\end{align*}
we have
$$
x^{l} \circ h = g^{*} \circ x^{l},
$$
provided that all $h_{i}$ are in $\mathbb{F}_{p^{n}}$. If $d > b$, we
have $m=bl+(d-b)$, so that $s=b$ and $k=d-b$.
Applying \ref{th:fifi}, we find $w = \sum_{0 \leq i \leq
  t}h_{i}x^{b-i}$ and $a=0$. Then
\begin{align*}
h & = x^{k}w(x^{l}),\\
g^{*}& = x^{k}w^{l}.
\end{align*}
Thus the example falls well within Ritt's Second Theorem. \cite{zan93}
points out that this was also remarked by Andrzej Kondracki, a student of
Andrzej Schinzel.
\end{example}

\section{Reducing vanishing to nonvanishing derivatives}

%%% 4
A particular strength of Zannier's and Schinzel's result in \ref{RST}
is that, contrary to earlier versions, the characteristic of $F$
appears only very mildly, namely in \ref{eq:RST-2}. We now elucidate
the case excluded by \ref{eq:RST-2}, namely $g'(g^{*})'=0$.
\cite{zan93} proves his version (Fact 3.3) of Ritt's Second Theorem
under the assumption \ref{eq:RST-2}. On his page 178, he writes ``We
finally remark that it is easy to obtain a version of [his] Theorem 1
valid also in case $g'$, say, vanishes. [...] Since however this
extension is quite straightforward we have decided not to discuss it
here in full detail, also in order not to complicate further the
already involved statement.''

The following executes this extension in the normal form. The special
case is reduced to the situation where both derivatives are nonzero in
\ref{fact}. True to Zannier's words, its statement is involved, and
the short version is: if \ref{eq:RST-2} is violated, remove the
component $x^{p}$ from the culprit as long as you can. Then
\ref{th:fifi} applies. We start with simple facts about $p$th powers.

\begin{lemma}\label{lem:propdivi}
  Let $F = \mathbb{F}_{q}$ be a finite field of characteristic $p$,
  let $l$, $m \geq 2$ be integers for which $p$ divides $n= lm$, and
  let $g$ and $h$ in $F[x]$ have degrees $l$ and $m$,
  respectively. Then the following hold.
  \begin{enumerate}
  \item\label{lem:propdivi-1} $g \circ h \in F[x^{p}]
    \Longleftrightarrow g'h'=0 \Longleftrightarrow g \in
    F[x^{p}]$ or $h \in F[x^{p}]$,
  \item\label{lem:propdivi-2} $\#(D_{\dg} \cap F[x^p])  = q^{\dg/p-1}$,
  \item\label{lem:propdivi-3}
    \begin{equation*}
      \#D_{\dg,l} \cap F[x^p] =
            \begin{cases}
        \#D_{\dg/p,l/p} & \text{if } p \mid l,\\
        \#D_{\dg/p,l} & \text{if } p \mid m.
      \end{cases}
    \end{equation*}
  \end{enumerate}
\end{lemma}
\begin{proof}
  \ref{lem:propdivi-1} is clear. For \ref{lem:propdivi-2}, all monic original
  Frobenius compositions are of the form $g \circ x^{p}$ with
  $g \in P_{\dg/p}$, and $g$ is uniquely determined by the
  composition. In \ref{lem:propdivi-3}, if  $p \mid
  l$ and according to \ref{eq:frob}, any $g \circ h \in
     \#D_{\dg,l} \cap F[x^p] $ can be uniquely rewritten as $x^p \circ g^* \circ h$
      with $g^{*} \in P_{l/p}$ and $h\in P_m$. If $p \mid m$, then the
  corresponding argument works.
  In this lemma, we use the otherwise undefined $D_{p,p} = \{x\}$
  (when $p = l = m$) and $D_{n/p,1} = P_{n/p}$ (when $p=l$).\hspace{-2ex}~
  \footnote{The lemma corrects the statements of the same lemma in the version
 published in \emph{Finite Fields and Their Applications.}}
  \end{proof}

We recall the Frobenius power $\varphi_{j} \colon F[x]\rightarrow
F[x]$ from \ref{rem:coll}.
\begin{lemma}\label{fact}
  In the above notation, assume that $(l,m,g,h,g^{*},h^{*})$ and $f$
  satisfy \ref{eq:RST-1}, \ref{eq:RST-3}, and \ref{eq:intersec}, and
  that $F$ is perfect.
  \begin{enumerate}
  \item\label{fact-0} The following are equivalent:
    \begin{enumerate}
    \item\label{fact-0-a} $f $ is a Frobenius composition,
    \item\label{fact-0-b} $ f' = 0$,
    \item $g'(g^{*})' = 0$,
    \item $g'h'(g^{*})'(h^{*})'= 0$.
    \end{enumerate}
  \item\label{fact-1} If $g'=0$, then $p \nmid l$ and $(g^{*})' \neq
    0$, and there exist positive integers $j$ and
    $M$, and monic original $G$, $G^{*}$, $H^{*} \in F[x]$ so that
    \begin{align}\label{al:fact}
      \begin{aligned}
        & m =p^{j}M, \deg G = \deg H^{*} = M, \deg G^{*}=l,\\
       & g= x^{p^{j}} \circ G, g^{*} \circ x^{p^{j}} = x^{p^{j}} \circ
       G^{*}, h^{*} = x^{p^{j}} \circ H^{*},\\
        & G' (G^{*})' \neq 0, G \circ h=G^{*} \circ H^{*},\\
        & f = x^{p^{j}} \circ G \circ h = x^{p^{j}} \circ (G^{*} \circ
        H^{*}).
      \end{aligned}
    \end{align}
    In particular, $(l, M, G, h, G^{*}, H^{*})$ satisfies
    \ref{eq:RST-1} through \ref{eq:RST-3} if $M > l$, and $(M,l,G^{*},
    H^{*}, G,h)$ does if $2 \leq M < l$. If $M = 1$, then $G$ and
    $H^{*}$ are linear.
  \item\label{fact-2} If $(g^{*})' = 0$, then $p \nmid m$ and $g' \neq
    0$, and there exist positive integers $d$ and $L$, and monic
    original $G,H, G^{*} \in F[x]$ with
    \begin{align}\label{eq:fgx}
      l&=p^{d}L, p \nmid L, g = \varphi_{d}(G), h = x^{p^{d}}\circ H,
      g^{*}=
      x^{p^{d}} \circ G^{*},\nonumber\\
      G'(G^{*})' &\neq 0, G \circ H = G^{*} \circ h^{*}, f =
      x^{p^{d}} \circ G \circ H.\nonumber
    \end{align}
    In particular, $(L,m, G, H, G^{*}, h^{*})$ satisfies
    \ref{eq:RST-1} through \ref{eq:RST-3} if $L \geq 2$.
  \item\label{fact-4} The data derived in \short\ref{fact-1} and
    \short\ref{fact-2} are uniquely determined. Conversely, given such
    data, the stated formulas yield $(l,m,g,h,g^{*}, h^{*})$ and $f$
    that satisfy \ref{eq:RST-1}, \ref{eq:RST-3}, and \ref{eq:intersec}.
  \end{enumerate}
\end{lemma}

\begin{proof}
  \short\ref{fact-0} (a) means that $f = x^{p} \circ G$ for
  some $G \in F[x]$, which is equivalent to (b). We have
  \begin{equation}\label{eq:g*}
    f' = (g' \circ h) \cdot h' = ((g^{*})' \circ h^{*}) \cdot (h^{*})'.
  \end{equation}
  If (b) holds, then $p\mid \deg f = \dg = lm$, hence $p \mid l$ or
  $p \mid m$. In the case $p \mid l$, \ref{eq:RST-1} implies that $p
  \nmid m$ and $g'(h^{*})' \neq 0$, hence $h'= (g^{*})'=0$ by
  \ref{eq:g*}. Symmetrically, $p \mid m $ implies that $g' =
  (h^{*})'=0$, so that (b) $\Longrightarrow$ (c) in both cases. (c)
  $\Longrightarrow$ (d) is trivial, and \ref{eq:g*} shows (d)
  $\Longrightarrow$ (b).

    \short\ref{fact-1} Let $j \geq 1$ be the largest integer for which
  there exists some $G \in F[x]$ with $g = x^{{p^{j}}} \circ G$. Then
  $j$ and $G$ are uniquely determined, $G$ is monic and original, $G'
  \neq 0$, $p^{j} \mid m$, $\deg G =mp^{-j}=M$, and $p \nmid l$ by
  \ref{eq:RST-1}. Furthermore, we have
  \begin{equation}\label{eq:Gh}
    g^{*} \circ h^{*} = g \circ h = x^{p^{j}}\circ G \circ h.
  \end{equation}

  Writing $h^{*} = \sum_{1 \leq i \leq m} h_{i}^{*}x^{i}$ with
  $h_{m}^{*}= 1$, we let $I=\{i \leq m \colon h^{*}_{i} \neq 0\}$ be
  the support of $h^{*}$. Assume that there is some $i \in I$ with
  $p^{j} \nodiv i$, and let $k$ be the largest such $i$. Then $k < m$,
  $m(l-1)+k$ is not divisible by $p^{j}$, the coefficient of
  $x^{m(l-1)+k}$ in $(h^{*})^{l}$ is $lh^{*}_{k}$, and in $g^{*} \circ
  h^{*}$ it is also $lh^{*}_{k}\neq 0$. This contradicts \ref{eq:Gh},
  so that the assumption is false and $h^{*}= x^{p^{j}} \circ H^{*}$ for a
  unique monic original $H^{*} \in F[x]$, of degree $M =m p^{-j}$.

  Setting $G^{*}=\varphi_{j}^{-1}(g^{*})$, we have $\deg G^{*} = \deg
  g^{*}=l$ and hence $(G^{*})' \neq 0$, $ x^{p^{j}}\circ G^{*}=
  \varphi_{j} (G^{*})\circ x^{p^{j}}$, and
  $$
  x^{p^{j}}\circ G \circ h=g \circ h= f = g^{*} \circ h^{*}=
  \varphi_{j}(G^{*})\circ x^{p^{j}} \circ H^{*}=x^{p^{j}} \circ G^{*}
  \circ H^{*},
  $$
  $$
  G\circ h= G^{*} \circ H^{*}.
  $$

  \short\ref{fact-2} Since $p \mid l = \deg g^{*}$, \ref{eq:RST-1}
  implies that $p \nmid m$, $g' \neq 0$, and $g' \circ h \neq 0$. In
  \ref{eq:g*}, we have $f'=0$ and hence $h' =0$. There exist monic
  original $G_{1}$, $H_{1} \in F[x]$ with $g^{*}= x^{p} \circ G_{1}$,
  $h = x^{p} \circ H_{1}$, and
  \begin{align*}
    x^{p} \circ G_{1} \circ h^{*}  &= f= g \circ x^{p}\circ
    H_{1}=x^{p}
    \circ \varphi_{1}^{-1}(g) \circ H_{1},\\
    G_{1} \circ h^{*}  &= \varphi_{1}^{-1}(g) \circ H_{1}.
  \end{align*}
  If $G_{1}' = 0$, then $H'_{1}=0$ and we can continue this
  transformation. Eventually we find an integer $j \geq 1$ and monic
  original $G_{j}$, $H_{j} \in F[x]$ with $p^{j} \mid l$, $g^{*}=
  x^{p^{j}} \circ G_{j}$, $h = x^{p^{j}} \circ H_{j}$, and $G_{j}'
  \neq 0$. We set $G = \varphi_{j}^{-1}(g), G^{*}= G_{j}$, and
  $H=H_{j}$. Then $G'(G^{*})'\neq 0,\; \deg G^{*} = \deg H = lp^{-j}, \deg G
  = m$. As above, we have
  \begin{align}\label{al:HH}
 f = (x^{p^{d}} \circ G^{*}) \circ h^{*} = g \circ (x^{p^{d}} \circ
  H) = x^{p^{d}} \circ G \circ H, \nonumber\\
G^{*}\circ h^{*}= G_{j} \circ h^{*} = \varphi_{j}^{-1}(g) \circ
  H_{j}= G \circ H.
   \end{align}

  In \short\ref{fact-2}, $d$ is defined as the multiplicity of $p$ in
  $l$. We now show that $j=d$. We set $l^{*} = lp^{-j}$. If $l^{*}
  \geq 2$, then the collision \ref{al:HH} satisfies the assumptions
  \ref{eq:RST-1} through \ref{eq:RST-3}, with $l^{*} < m$ instead of
  $l$. Thus \ref{th:fifi} applies.
  In the First Case, \ref{eq:unidet} shows
  that $p \nmid l^{*}$. It follows that $j = d$ and $l^{*}=L$. In the
  Second Case, we have $p \nmid l^{*}m = lp^{-j}m$ by \ref{eq:ab-5}
  so that again $j=d$ and $l^{*}=L$. In the remaining case $l^{*}=1$,
  we have $L=1$ and $G^{*}=H=x$.

  \short\ref{fact-4} The uniqueness of all quantities is clear.
\end{proof}

\section{Normal form: vanishing derivatives and coprime degrees}

\ref{fact} allows us to get rid of the assumption \ref{eq:RST-2}, namely
that $g'(g^{*})' \neq 0$, in the normal form of \ref{th:fifi}.
We need some simple properties of the Frobenius
map $\varphi_{j}$ from \ref{rem:coll}.

\begin{lemma}\label{lem:Frobmap}
  Let F be a field of characteristic $p\geq2$, $f$, $g\in F[x]$, $\parTwo\in
  F$, let $ i$, $j \geq 1$, and denote by $f'$ the derivative of $f$. Then
  \begin{enumerate}
  \item \label{lem:Frobmap-1}
    $\varphi_{j}(fg)=\varphi_{j}(f)\varphi_{j}(g)$ and
    $\varphi_{j}(f^{i})={\varphi_{j}(f)}^{i}$,
    \item\label{lem:Frobmap-3} $\varphi_{j}(f\circ
    g)=\varphi_{j}(f)\circ \varphi_{j}(g)$ and
    $\varphi_{j}(f(a))=\varphi_{j}(f)(\parTwo^{p^{j}})$,
  \item\label{lem:Frobmap-5} $\varphi_{j}(f')=\varphi_{j}(f)'$.
  \end{enumerate}
\end{lemma}

\begin{proof}\ref{lem:Frobmap-1} is immediate. For \ref{lem:Frobmap-3}, we write
  $f=\sum{f_{i}x^{i}}$ with all $ f_{i}\in F$. Then
 $$\varphi_{j}(f\circ
 g)=\varphi_{j}(\sum{f_{i}g^{i}})=\sum{f_{i}^{p^{j}}\varphi_{j}(g^{i})}=
 \varphi_{j}(f)\circ\varphi_{j}(g).
$$
The second claim is a special case. For \ref{lem:Frobmap-5}, we have
  $$
  \varphi_{j}(f{'})=\varphi_{j}(\sum{if_{i}x^{i-1}})=\sum{i^{p^{j}}f_{i}^{p^{j}}x^{i-1}}
 =\sum{if_{j}^{p^{j}}x^{i-1}}
  =\varphi_{j}(f)'.  \qed $$
\end{proof}

%%cp 1
\begin{theorem}\label{lem:div2a}
  Let $F$ be a perfect field of characteristic $p \geq 2$. Let $m > l \geq 2$
  be integers with $\gcd(l,m)=1$, set $\dg = lm$ and let $f,g,h,g^{*},
  h^{*} \in F[x]$ be monic original of degrees $n$, $m$, $l$, $l$,
  $m$, respectively, with $f=g \circ h = g^{*} \circ h^{*}$. Then the
  following hold.
  \begin{enumerate}
  \item\label{lem:div2a-i} If $g'=0$, then there exists a uniquely
    determined positive integer $j$ so that $p^{j}$ divides $m$ and
    either (\bare\ref{lem:div2a-i/1}) or (\bare\ref{lem:div2a-i/2})
        hold; furthermore, $p \nmid l$ and (\bare\ref{lem:div2a-i/3}) is
    true. We set $M=p^{-j}m$.
    \begin{enumerate}
    \item\label{lem:div2a-i/1} (First Case) Exactly one of the
      following three statements is true.
      \begin{enumerate}
      \item $M>l$ and there exist a monic $W \in F[x]$ of
        degree $S = \lfloor M/l \rfloor$ and $\parTwo \in F$ so that
        $$
        KW +lxW'\neq 0
        $$
        for $K = M-l S$, and \ref{eq:mopo} through \ref{eq:unidet-3}
        hold for $k = p^{j}K$, $s=p^{j}S$, and $w=W^{p^{j}}$.
        Conversely, any $W$ and $a$ as above yield via these formulas
        a collision satisfying \ref{eq:RST-1}, \ref{eq:RST-3} and
        \ref{eq:intersec}, with $g'=0$. If $p \nmid M$, then $W$ and
        $a$ are uniquely determined by $f$ and $l$.
      \item $2 \leq M <l$ and there exist a monic $W \in F[x]$ of
        degree $S = \lfloor l/M \rfloor$ and $\parTwo \in F$ so that
        \begin{align*}
          & KW +lxW'\neq 0,\\
         & f  = \big(x^{kM} w^{M}(x^{M})\big)^{[a]}
        \end{align*}
        for $K = l-MS$, $k = p^{j}K$, $s=p^{j}S$, and $w=W^{p^{j}}$,
        and \ref{eq:mopo} through \ref{eq:unidet-3} hold with $l$
        replaced by $M$. Conversely, any $W$ and $a$ as above yield
        via these formulas a collision satisfying \ref{eq:RST-1},
        \ref{eq:RST-3} and \ref{eq:intersec}, with
        $g'=0$. Furthermore, $W$ and $a$ are uniquely determined by
        $f$ and $l$.
      \item $m=p^{j}$, $g=h^{*}= x^{p^{j}}$, and $g^{*}=
        \varphi_{j}(h)$.
      \end{enumerate}
    \item\label{lem:div2a-i/2} (Second Case) $p \nmid M$ and all
      conclusions of \ref{th:fifi-2} hold, except \ref{eq:ab-5}.
    \item\label{lem:div2a-i/3}
      Assume that $M \geq 2$, and let $f$ be a collision of the Second
      Case. Then $f$ belongs to the First Case if and only if $~min\{l,M\}=2$.
    \end{enumerate}
  \item\label{lemdiv2a-ii} If $(g^{*})'=0$, then there exists a unique
    positive integer $d$ such that $p^{d} \mid l$, $p \nmid
    p^{-d}l=L$, and either (\bare\ref{lemdiv2a-ii/1}) or
    (\bare\ref{lemdiv2a-ii/2}) holds; furthermore,
    (\bare\ref{lemdiv2a-ii/3}) is true.
\begin{enumerate}
\item\label{lemdiv2a-ii/1}(First Case) There exist a monic $ w \in
  F[x]$ of degree $\lfloor m/L\rfloor$ and $\parTwo \in F $ so that
    \begin{align*}
      f& = \big( x^{kl} w^{L}(x^{\ell})\big)^{[a]} ,\\
      (g,h) & = (x^{k} w^{L}, x^{\ell})^{[a]}, \\
      (g^{*}, h^{*}) & = \big(x^{\ell}, x^{k}\varphi_{d}^{-1}(w)(x^{L})\big)^{[a]} ,\\
      kw +Lxw' &\neq 0,
    \end{align*}
    where $m = L \lfloor m/L\rfloor + k$. The quantities $w$ and $a$
    are uniquely determined by $f$ and $l$. Conversely, any $w$ and
    $a$ as above yield via these formulas a collision satisfying
    \ref{eq:RST-1}, \ref{eq:RST-3}, and
    \ref{eq:intersec}.

\item\label{lemdiv2a-ii/2} (Second Case) There exist $z,a\in F$ with
    $z\neq 0$ for which all conclusions of \ref{th:fifi-2} hold, except
    \ref{eq:ab-5}. Conversely, any $(z,a)$ as above yields a collision
    satisfying \ref{eq:RST-1}, \ref{eq:RST-3} and \ref{eq:intersec}.
     \item\label{lemdiv2a-ii/3} When $L\geq 3$, then
       (\bare\ref{lemdiv2a-ii/1}) and (\bare\ref{lemdiv2a-ii/2}) are
       mutually exclusive. For $L\leq 2$, (\bare\ref{lemdiv2a-ii/2}) is
       included in (\bare\ref{lemdiv2a-ii/1}).
  \end{enumerate}
\end{enumerate}
\end{theorem}

\begin{proof}
  \short\ref{lem:div2a-i} We take the quantities $j$, $M$, $G$,
  $G^{*}$, $H^{*}$ from \ref{fact-1} and apply \ref{th:fifi} to the
  collision $G\circ h=G^{*}\circ H^{*}$ in \ref{al:fact}. We start
  with the First Case (\ref{th:fifi-1}). If $M > l$, it yields a monic
  $W \in F [x]$ of degree $S =\lfloor M/l \rfloor$ and $\parTwo \in F$
  with
  \begin{align}
    &G\circ h=G^{*}\circ h^{*}=(x-\parTwo^{*})\circ x^{Kl}W^{l}(x^{l})\circ (x+\parTwo), \nonumber \\
    & KW + lxW' \neq 0,
  \end{align}
  where $K=M-l S$ and $\parTwo^{*}=\parTwo^{Kl}W^{l}(\parTwo^{l})$.  We
  set $k=p^{j}K$ and $w=W^{p^{j}}$. Then
  \begin{align*}
    f&=g \circ h = x^{p^{j}}\circ G \circ h =
    x^{p^{j}}\circ(x-\parTwo^{*})
    \circ x^{Kl}W^{l}(x^{l})\circ(x+\parTwo)\\
    &=\bigl( x-(\parTwo^{*})^{p^{j}}\bigr) \circ
    x^{p^{j}Kl}(W^{p^{j}})^{l}(x^{l})\circ(x+\parTwo) =
    \big(x^{kl}w^{l}(\parTwo^{l})\big)^{[a]}.
  \end{align*}
  The formulas \ref{eq:unidet-1} and \ref{eq:unidet-3} are readily
  derived in a similar fashion. Furthermore, we have
$$
ls + k = lp^{j} S +p^{j}(M-lS)=m,
$$
and the other claims also follow.

If $2 \leq M < l$, we have to reverse the roles of $M$ and $l$ in the
application of \ref{th:fifi-1}. Thus we now find a monic $W \in F[x]$
of degree $S= \lfloor l/M \rfloor$ and $\parTwo\in F$ with
  $$
  G \circ h = (x-a^{*}) \circ x^{KM}W^{M}(x^{M}) \circ (x+\parTwo),
  $$
  where $K = l-M S$, $a^{*}=\parTwo^{KM}W^{M}(\parTwo^{M})$, and $KW
  +MxW' \neq 0$. We set $k=p^{j}K$ and $w=W^{p^{j}}$. Then
  \begin{align*}
    f  &= x^{p^{j}} \circ G \circ h = \varphi_{j}(x-a^{*}) \circ
    x^{p^{j}} \circ x^{KM}W^{M}(x^{M}) \circ (x+\parTwo)\\
   &= \big(x^{kM} w^{M}(x^{M})\big)^{[a]}.
  \end{align*}
Furthermore we have
$$
Ms+k = Mp^{j} S + p^{j}(l-M S) = p^{j}l.
$$
Equations \ref{eq:unidet-1} through \ref{eq:unidet-3} are readily
checked, the claim about the converse is clear, and since $p \nmid l$,
$W$ and $a$ are uniquely determined.

If $M=1$, then $g= h^{*}=x^{p^{j}}$, $f=x^{p^{j}} \circ h =
\varphi_{j}(h) \circ x^{p^{j}}= g^{*} \circ x^{p^{j}}$, and $g^{*}=
\varphi_{j}(h)$ by the uniqueness of tame decompositions..

  In the Second Case of \ref{th:fifi}, we use $T_{p^j} = x^{p^j}$ from
  \ref{Tsqfree-3}. Now \ref{th:fifi-2} provides $z, \parTwo \in F$
  with $z\neq 0$ and
  \begin{align*}
  G \circ h &= G^{*}\circ H^{*}= (x-T_{lM}(\parTwo,z)) \circ
  T_{lM}(x,z)\circ(x+\parTwo),\\
  G &= (x-T_{lM}(a,z) \circ T_{M}(x,z^{l}) \circ (x+T_{l}(a,z)).
\end{align*}
Since $G' \neq 0$, we have $p\nmid M$, and hence $p \nmid lM$. Thus
$z$ and $a$ are uniquely determined. Furthermore
\begin{align*}
  f&=g\circ h = x^{p^{j}} \circ G\circ h
  = (x^{p^{j}}-(T_{lM}(\parTwo,z))^{p^{j}})
  \circ T_{lM}(x,z)\circ (x+\parTwo)\\
  &=(x- T_{\dg}(\parTwo,z))\circ x^{p^{j}}\circ T_{lM}(x,z)\circ
  (x+\parTwo) = T_{\dg}(x,z)^{[a]}.
  \end{align*}

  In \short\ref{lem:div2a-i/3}, we have $p \nmid lM = p^{-j}\dg$ from
 (\bare\ref{lem:div2a-i/2}). By \ref{th:fifi-3}, $G \circ h$
  belongs to the First Case if and only if $~min \{l,M\}=2$.

\short\ref{lemdiv2a-ii} We take $d$, $L$, $G$, $H$, $G^{*}$ from
\ref{fact-2}, and apply \ref{th:fifi} to the collision $G \circ H =
G^{*} \circ h^{*}$. In the First Case, this yields a monic $W \in
F[x]$ of degree $\lfloor m/L \rfloor$ and $\parTwo \in F$ so that the
conclusions of \ref{th:fifi-1} hold for these values, with $k =
m-L \cdot \lfloor m / L \rfloor$ and $kW + LxW' \neq 0$. We set $w =
\varphi_{d}(W)$. Then
  \begin{align*}
    \deg G &= \deg (x^{k}W^{L})= (m-L\cdot \lfloor m / L
    \rfloor) + L \cdot \lfloor m / L \rfloor = m,\\
    g &= \varphi_{d}(G) = \varphi_{d}\bigl
    ((x-\parTwo^{kL}W^{L}(\parTwo^{L}))\circ x^{k}W^{L}
    \circ (x+\parTwo^{L})\bigr )\\
    &= \varphi_{d}(x-\parTwo^{kL}W^{L}(\parTwo^{L})) \circ
    \varphi_{d}(x^{k}W^{L})
    \circ \varphi_{d}(x+\parTwo^{L})\\
    &=(x-\parTwo^{kl}w^{L}(\parTwo^{l})) \circ x^{k}w^{L} \circ
    (x+\parTwo^{l}) = (x^{k} x^{L})^{[a^{\ell}]} .\\
    h &= x^{p^{d}} \circ H = x^{p^{d}} \circ (x-\parTwo^{L}) \circ
    x^{L}
    \circ  (x+\parTwo)= (x^{\ell})^{[a]},\\
    g^{*} &= x^{p^{d}} \circ G^{*} = x^{p^{d}} \circ
    (x-\parTwo^{kL}W^{L}(\parTwo^{L}))
    \circ x^{L} \circ (x+\parTwo^{k}W(\parTwo^{L}))\\
    &= (x-\parTwo^{kl} W^{p^{d}L}(\parTwo^{L})) \circ x^{l} \circ
    (x+\parTwo^{k}W(\parTwo^{L}))\\
    &= (x-\parTwo^{kl}w^{L}(\parTwo^{l}))\circ x^{l} \circ (x
    + \parTwo^{k} \varphi_{d}^{-1}(w)(\parTwo^{L})) =
    (x^{\ell})^{[\parTwo^{k} \varphi_{d}^{-1}(w)(\parTwo^{L})]} ,\\
    h^{*} &= (x-\parTwo^{k}W(\parTwo^{L})) \circ x^{k}W(x^{L}) \circ (x+\parTwo)\\
    &= (x-\parTwo^{k} \varphi_{d}^{-1}(w)(\parTwo^{L})) \circ x^{k}
    \varphi_{d}^{-1}(w)(x^{L}) \circ (x+\parTwo) = \big(x^{k}
    \varphi_{d}^{-1}(w)(x^{L})\big)^{[a]},\\
    f & = \big(x^{kl}w^{L}(x^{l})\big)^{[a]} .
  \end{align*}

Furthermore, \ref{lem:Frobmap} implies that
$$
kw + Lxw'= k \varphi_{d}(W)+Lx \varphi_{d}(W)'=
\varphi_{d}(kW+LxW')\neq 0.
$$
The claimed uniqueness follows, with a small modification, by the
argument for \ref{th:fifi-1}. With notation in the spirit of
\ref{eq:unique}, one finds that
\begin{align*}
x^{kL}\tilde w^{L}(x^{L})\circ x^{p^{d}}&=u\circ
x^{kL}w^{L}(x^{L})\circ(x+a^{p^{d}}-\tilde a^{p^{d}})\circ
x^{p^{d}},\\
x^{kL}\tilde w^{L}(x^{L})&=u\circ
x^{kL}w^{L}(x^{L})\circ(x+a^{p^{d}}-\tilde a^{p^{d}}).
\end{align*}
The degree $p^{-l}\dg$ of these polynomials is not divisible by $p$,
and the remaining argument following \ref{eq:unique} applies.

In the Second Case, \ref{th:fifi-2} provides $z,a\in F$ with $z\neq 0$
and
\begin{align*}
  g&=\varphi_{d}(G)=\varphi_{d}\bigl((x-T_{mL}(a,z))\circ
  T_{m}(x,z^{L})\circ(x+T_{L}(a,z))\bigr)\\
  &=\bigl(x-\varphi_{d}(T_{mL}(a,z))\bigr)\circ
  \varphi_{d}(T_{m}(x,z^{L}))\circ\bigl(x+\varphi_{d}(T_{L}(a,z))\bigr)\\
  &=(x-T_{mL}(a,z)^{p^{d}})\circ
  T_{m}(x,(z^{L})^{p^{d}})\circ(x+T_{L}(a,z)^{p^{d}})\\
  &=(x-T_{\dg}(a,z))\circ T_{m}(x,z^{l})\circ (x+T_{l}(a,z)) =
  T_{m}(x,z^{\ell})^{[T_{\ell}(a,z)]},\\
  h&=x^{p^{d}}\circ H=x^{p^{d}}\circ(x-T_{L}(a,z))\circ
  T_{L}(x,z)\circ(x+a) \\
  &=(x-T_{l}(a,z))\circ T_{l}(x,z)\circ (x+a)   = T_{\ell}(x,z)^{[a]},\\
  g^{*}&=x^{p^{d}}\circ G^{*}=x^{p^{d}}\circ(x-T_{Lm}(a,z))\circ
  T_{L}(x,z^{m})\circ(x+T_{m}(a,z))\\
  &=(x-T_{\dg}(a,z))\circ x^{p^{d}}\circ T_{L}(x,z^{m})\circ
  (x+T_{m}(a,z))\\
  &= x-T_{\dg}(a,z))\circ T_{l}(x,z^{m})\circ (x+T_{m}(a,z))= T_{\ell}
  (x,z^{m})^{[T_{m}(a,z)]},\\
  h^{*}& = T_{m}(x,z)^{[a]},\\
  f& =  T_{\dg}(x,z)^{[a]}   .
\end{align*}

\short\ref{lemdiv2a-ii/3} follows from \ref{th:fifi-3} for $L\geq
2$. If $L=1$, then $l=p^{d}$ and $k=0$ in
\short\ref{lemdiv2a-ii/1}. For any $f= T_{\dg}(x,z)^{[a]}$ in
\short\ref{lemdiv2a-ii/2}, we take $w=T_{m}(x,z^{p^{d}})$. Then
\begin{align*}
T_{\dg}(x,z)&=T_{m}(x,z^{p^{d}})\circ T_{p^{d}}(x,z)=w\circ
x^{p^{d}},
\end{align*}
so that $f = w(x^{\ell})^{[a]}$ is an instance of \short\ref{lemdiv2a-ii/1}.
\end{proof}

\section{Arbitrary tame degrees}

If $p \nmid \dg$, then the case where $\gcd (l,m) \neq 1$ is reduced
to the previous one by the following result of \cite{tor88a}. We will
only use the special case where $l=l^{*}$ and $m = m^{*}$.

\begin{fact}
\label{tortrat}
Suppose we have a field $F$ of characteristic $p \geq 0$, integers $l,
l^{*}, m, m^{*} \geq 2$ with $p \nmid lm$, and monic original
polynomials $g, h, g^{*}, h^{*} \in F[x]$ of degrees $m, l, l^{*},
m^{*}$, respectively, with $g \circ h = g^{*} \circ
h^{*}$. Furthermore, let $i = \gcd (m, l^{*})$ and $j=\gcd
(l,m^{*})$. Then the following hold.
\begin{enumerate}
\item\label{fact:monopo-1}
There exist unique monic original polynomials $u, v, \tilde{g}, \tilde{h},
\tilde{g}^{*}, \tilde{h}^{*} \in F[x]$ of degrees $i, j, m/i, l/j,
l^{*}/i, m^{*}/j$, respectively, so that
\begin{align}\label{al:monopo}
g &= u \circ \tilde{g},\nonumber\\
h&= \tilde{h} \circ v,\\
g^{*}& = u \circ \tilde{g}^{*},\nonumber \\
h^{*}&= \tilde{h}^{*} \circ v.\nonumber
\end{align}
\item\label{fact:monopo-2}
Assume that $l = l^{*} < m= m^{*}$. Then $i=j$ and $m/i, l/i,
 \tilde{f}=\tilde{g} \circ \tilde{h},
 \tilde{g}, \tilde{h},
\tilde{g}^{*}, \tilde{h}^{*}$ satisfy the assumptions of \ref{th:fifi}.
\end{enumerate}
\end{fact}
\begin{proof}
  \short\ref{fact:monopo-1} \cite{tor88a}, Proposition 1, proves the
  claim if $F$ is algebraically closed, but without the condition of
  being monic original. Thus we have four decompositions
  \ref{al:monopo} over an algebraic closure of $F$. We may choose all
  six components in \ref{al:monopo} to be monic original. They are
  then uniquely determined. Since $p \nmid \dg$, decomposition is
  rational; see \cite{sch00c}, I.3, Theorem 6, and \cite{kozlan89} or
  \cite{gat90c} for an algorithmic proof. It follows that the six
  components are in $F[x]$.

\short\ref{fact:monopo-2} We have $\gcd(l/i, m/i)=1$, and
$$
f = (u \circ \tilde{g}) \circ (\tilde{h} \circ v) = (u \circ
\tilde{g}^{*})\circ (\tilde{h}^{*} \circ v).
$$
The uniqueness of tame decompositions with
prescribed component degrees implies that $\tilde{g} \circ
\tilde{h} = \tilde{g}^{*} \circ \tilde{h}^{*}$. The other requirements
are immediate.
\end{proof}

\cite{ziemue08} show that \ref{tortrat} holds over $\mathbb C$ and
mention that their proof also works under the conditions as stated.

\cite{dorwha74} exhibit the example
$$
\bigl ( x^{p+1} \circ (x^p+x) \bigr ) \circ (x^p-x) =
x(x^{p-1}-1)^{p+1} \circ x^{p+1},
$$
which violates both assumption and conclusion of Ritt's First Theorem
on complete decompositions.
Both left components have nonzero derivative, the gcd of their degrees equals $p$,
 and the one in the right-hand decomposition is indecomposable.
 Thus no $u$ as in \ref{al:monopo} exists and
 the conclusions of \ref{fact:monopo-1}  may fail under the weaker assumption
that $g' (g^*)' \neq 0$.
Composing both right components with $x^p$ on the right
yields a counterexample to the conclusion of \ref{fact:monopo-2}.

\citeauthor{tor88a}'s result, together with the preceding material,
determines $D_{\dg,l} \cap D_{\dg,m}$ exactly if $p \nmid \dg =
lm$.
\begin{theorem}\label{thm:FFC}
  Let $\mathbb{F}_{q}$ be a finite field of characteristic $p$, and
  let $m > l \geq 2$ be integers with $p \nmid \dg = lm$, $i =
  \gcd(l,m)$ and $s=\lfloor m/l \rfloor$.  Let $t=\#(D_{\dg,l} \cap
  D_{\dg,m})$. Then the following hold.
\begin{enumerate}
\item\label{thm:FFC-1}\begin{align*}
 t=
\begin{cases}
q^{2l+s-3} & \text{if } l \mid m,\\
q^{2i}(q^{s-1}+(1- \delta_{l,2})
(1-q^{-1})) & \text{otherwise}.
\end{cases}
\end{align*}
\item\label{thm:FFC-2}
$$
t\leq q^{2l+s-3}.
$$
\end{enumerate}
\end{theorem}
\begin{proof}
  \short\ref{thm:FFC-1} Let $T = D_{\dg,l} \cap D_{\dg,m}$ and $U =
  D_{\dg/i^{2}, l/i} \cap D_{\dg/i^{2}, m/i}$. Then
  \ref{fact:monopo-2} implies that $T = P_{i} \circ U \circ
  P_{i}$, using $G \circ H = \{g \circ h \colon g \in G, h \in
  H\}$ for sets $G, H \subseteq F[x]$. Furthermore, the composition
  maps involved are injective. Thus
  $$
  \#T = (\#P_{i})^{2}\cdot\# U = q^{2i-2} \cdot\#U.\\
  $$
If $l \nmid m$, then $l/i \geq 2$, $\gcd(l/i, m/i)=1$, and from
\ref{lem:div-1} we have
$$
\#U =q^{s+1}+(1-\delta_{l,2})(q^{2}-q),
$$
which implies the claim in this case. If $l \mid m$, then $l/i = 1$ and
\ref{lem:div-b} is inapplicable. Now
\begin{align*}
U &= D_{m/l,1}\cap D_{m/l, m/l}= P_{m/l},\\
\#U  &= \#P_{m/l}= q^{m/l-1}= q^{s-1},
\end{align*}
which again shows the claim.

\short\ref{thm:FFC-2} By  \ref{thm:FFC-1}, we may assume that $l \nmid
m$. Thus $l \geq 2i$. If $l=2$, the second summand in \ref{thm:FFC-1}
vanishes and $t \leq q^{2i+s-1} \leq q^{2l+s-3}$. We may now also
assume $l \geq 3$. Then $t \leq q^{2i}(q^{s-1}+1) \leq 2q^{2i+s-1}
\leq q^{2i+s}$, since $s \geq 1$. Furthermore, $2i+s \leq l+s \leq 2l+s-3$.
\end{proof}

This result shows \ref{thm:FFC-1-table} through \ref{thm:FFC-2-table}
in \ref{tab:cor-table}.  The equation in  \ref{lem:div-1-table} is a
special case of \ref{thm:FFC-1a-table}.

\section{Arbitrary degrees}

We now have determined the size of the intersection if either $p\nmid
\dg$ or $~gcd(l,m)=1$. It remains a challenge to do this with the same
precision when both
conditions are violated. The following approach yields rougher estimates.
\begin{theorem}\label{thm:div2a}
  Let $F$ be a field of characteristic $p\geq 2$, let $l, m ,\dg
    \geq 2$ be integers with $p\mid \dg = lm$, and set $T =D_{\dg, l}
  \cap D_{\dg,m} \setminus F[x^p] $.  Then the following hold.
  \begin{enumerate}
  \item\label{lem:div-6} If $p \nmid l$, then for any $f \in T$ there
    exist monic original $g^{*}$ and $h^{*}$ in
    $F[x]$    of degrees $l$ and $m$, respectively, with $f = g^{*} \circ
    h^{*}$, $(g^{*})'(h^{*})' \neq 0$, and $0\leq\deg(h^{*})' < m-l$.

  \item\label{lem:div-7} If $p \mid l$, then for any $f\in T$ there
    exist monic original $g$ and $h\in F[x]$ of degrees $m$ and $l$,
    respectively, with $f=g\circ h$ and $\deg g'\leq \ m-(m+1)/l$.
  \end{enumerate}
\end{theorem}
\begin{proof} We take a collision \ref{eq:intersec} and its derivative
  \ref{eq:g*}.  Since $f \not\in F[x^p]$, we have $f' \neq 0$.

  \short\ref{lem:div-6} Since $p\mid m$, we have $\deg g' \leq m-2$, $(h^{*})' \neq 0$, and
  $\deg(h^{*})' \geq 0$, so that
  \begin{align*}
    \dg - m + \deg(h^{*})' &=(l-1) \cdot m + \deg(h^{*})' = \deg
    f'\\
   &\leq (m-2)\cdot l+l-1 = \dg-l-1,\\
   0&\leq\deg(h^{*})' < m-l.
  \end{align*}

  \short\ref{lem:div-7} We have $g'h' \neq 0$,  $\deg(g^{*})'\leq l-2$, $\deg h' \geq 0$,
  and
  \begin{align*}
    l \cdot \deg g' &\leq l \cdot \deg g' + \deg h' = \deg f'
    \\
    &\leq (l-2)\cdot m + m - 1 = lm - m -1,
    \\
    \deg g' &\leq m - \frac{m+1}{l}.\qed
  \end{align*}
\end{proof}
We deduce the following upper bounds on $\#T$.

%%% label ge�ndert
\begin{corollary}%%\label{cor:div4}
  \label{cor:ffchar}
  Let $\mathbb{F}_{q}$ be a finite field of characteristic $p$ and
  $l$, $m$, $\dg \geq 2$ be integers with $p \mid \dg = lm$, and set
  $t =\#(D_{\dg, l} \cap D_{\dg,m}\setminus F[x^p])$.  Then the
  following hold.
  \begin{enumerate}
  \item\label{cor:ffchar-2} If $p \nmid l$, then
    $$
    t \leq q^{m+\lceil l/p \rceil-2}.
    $$
  \item %%%\label{cor:div-7}
    \label{cor:ffchar-3}
    If $p\mid l$ and $l < m$, we set $c=\lceil (m-l+1)/l\rceil$. Then
    $$
    t\leq q^{m+l-c+\lceil c/p\rceil-2}.
    $$
    If $l\mid m$, then $c=m/l$.
 \end{enumerate}
\end{corollary}
\begin{proof}
  \short\ref{cor:ffchar-2} Any $h^{*}$ permitted in \ref{lem:div-6}
  has nonzero coefficients only at $x^{i}$ with $p \mid i$ or $i \leq
  m-l$. Since $p\mid m$, the number of such $i$ is $m-l+\lceil l/p
  \rceil$. Taking into account that $h^{*}$ is monic, the number of
  $g^{*} \circ h^{*}$ is at most
  $$
  q^{l-1} \cdot q^{m-l+\lceil l/p \rceil-1}=
  q^{m+\lceil l/p \rceil-2}.
  $$

  \short\ref{cor:ffchar-3} The polynomials $g$ permitted in
  \ref{lem:div-7} are monic of degree $m$ and satisfy
  \begin{align*}
    \deg g' &\leq m - \frac{m+1}{l},\\
    \deg g' &\leq m-2.
  \end{align*}
  Thus $p\mid m$, and $g$ has nonzero coefficients only at $ x^{i}$
  with $i \leq m$ and $p \mid i$ or $1\leq i \leq m-c$. The number of
  such $i$ is $m-c+\lceil c/p\rceil$. By composing with $h$ on the
  right and using that $g$ is monic, we find
  $$
  t \leq  q^{m-c+ \lceil c/p\rceil-1}
  \cdot q^{l-1} = q^{m+l-c+ \lceil c/p \rceil-2}.
  $$
  If $l\mid m$, then $c=m/l-1+\lceil 1/l\rceil =m/l$.
\end{proof}
This shows \ref{cor:ffchar-2-table} and \ref{cor:ffchar-3-table} in
\ref{tab:cor-table}.

For perspective, we also note the following lower bounds on $\#T$ from
\cite{gat13,gat14}. Unlike the results up to \ref{thm:FFC}, there is a
substantial gap between the upper and lower bounds.

%%einfügen von Theorem 3.31 aus unidec, 11.09.09
\begin{fact}\label{th:decom}
  Let $\F_{q} $ have characteristic $p$ with $q=p^{e}$, and take
  integers $d \geq 1$,  $\degOne=ap^{d}$ with $p\nmid a$, $m
  \geq2$, $\dg=\degOne m$, $c= \gcd(d,e)$, $z=p^{c}$, $\mu =
  \gcd(p^{d}-1,m)$, and $r=(p^{d}-1)/\mu$. Then we have the following
  lower bounds on the cardinality of $ ~im \gamma_{\dg,\degOne}$.
  \begin{enumerate}
  \item\label{th:decom-1} If $p^{d}\neq m$ and $\mu=1$:
    $$
    q^{\degOne+m-2}(1-q^{-1}(1+q^{-p+2}
    \frac{(1-q^{-1})^{2}}{1-q^{-p}})) (1-q^{-\degOne}).
    $$
  \item\label{th:decom-2} If $p^{d}\neq m$:
    $$
    q^{\degOne+m-2}\bigl((1-q^{-1}(1+q^{-p+2}
    \frac{(1-q^{-1})^{2}}{1-q^{-p}}))(1-q^{-\degOne})
    $$
    $$
    -q^{-\degOne-r+2}\frac{(1-q^{-1})^{2}(1-q^{-r(\mu-1)})}{1-q^{-r}}(1+q^{-r(p-2)})
    \bigr).
    $$
    \end{enumerate}
\end{fact}

\begin{corollary}%%\label{cor:div4}
  \label{cor:ffcharb}
  Let $\mathbb{F}_{q}$ be a finite field of characteristic $p$, $l$ a
  prime number dividing $m >l$, assume that $p \mid \dg = lm$, and set
  $t =\#(D_{\dg, l} \cap D_{\dg,m}\setminus F[x^p])$.  Then the
  following hold.
  \begin{enumerate}
  \item\label{cor:ffcharb-4} If $p=l \mid m$ and each nontrivial
    divisor of $m / p$ is larger than $p$, then
    $$
    t \geq  q^{2p+m/p-3}(1-q^{-1})(1-q^{-p+1}).
    $$
  \item%%%\label{cor:div-6}
    \label{cor:ffcharb-1}
    If $p \neq l$ divides $m$ exactly $d \geq 1$ times, then
    \begin{equation}\label{eq:ffcharb-1}
    t \geq q^{2l+m/l-3}(1-q^{-m/l})(1-q^{-1}(1+q^{-p+2}
    \frac{(1-q^{-1})^{2}}{1-q^{-p}}))
    \end{equation}
    if $l \nmid p^{d}-1$. Otherwise we set $\mu=~gcd(p^{d}-1,l)$,
    $r=(p^{d}-1)/\mu$ and have
\begin{align}\label{al:ffcharb-1}
  \begin{aligned}
      t& \geq q^{2l+m/l-3}\bigl((1-q^{-1}(1+q^{-p+2}
      \frac{(1-q^{-1})^{2}}{1-q^{-p}}))(1-q^{-m/l})\\
      & \quad-q^{-m/l-r+2}
      \frac{(1-q^{-1})^{2}(1-q^{-r(\mu-1)})}{1-q^{-r}}
      (1+q^{-r(p-2)})\bigr).
    \end{aligned}
\end{align}
 \end{enumerate}
\end{corollary}

\begin{proof}
  \short\ref{cor:ffcharb-4} Clearly, $t$ is at least the number of $g
  \circ w \circ h$ with $g, w, h \in \mathbb{F}_{q}[x]$ monic original of degrees
  $p$, $m/p$, $p$, respectively.

We first bound the set $S$ of $h^{*}= w \circ h$ with $h^{*}_{m-1}
\neq 0$. We denote as $h_{p-1}$ the second highest coefficient of
$h$. Then $h^{*}_{m-1}= m/p \cdot h_{p-1}$, and $h^{*}_{m-1}$ vanishes
if and only if $h_{p-1}$ does. By the uniqueness of tame
decompositions, $\gamma_{m, m/p}$ is injective, so that
  $$
  \#S= q^{m/p-1} \cdot q^{p-1}(1-q^{-1})= q^{m/p+p-2}(1-q^{-1}).
  $$
  In the lower bound of \ref{th:decom-1} we replace the number
  $q^{m-1}(1-q^{-1})$ of all possible second components by $\#S$. Then
  \begin{align*}
  t& \geq q^{p+m-2}  (1-q^{-p})
  (1-q^{-1} (1+q^{-p+2}\frac{(1-q^{-1})^{2}}{1-q^{-p}}))
  \cdot\frac{\#S}{q^{m-1}(1-q^{-1})}\\
  & =q^{2p+m/p-3}(1-q^{-p})(1-q^{-1}(1+q^{-p+2}\frac{(1-q^{-1})^{2}}
  {1-q^{-p}}))\\
  &  =q^{2p+m/p-3}(1-q^{-1})(1-q^{-p+1}).
\end{align*}
  \short\ref{cor:ffcharb-1} For any monic original $g, w, h \in
  \mathbb{F}_{q}[x]$ of degrees $l, m/l, l$, respectively, we have $g
  \circ w \circ h \in D_{\dg,l} \cap D_{\dg,m}$. We now
  estimate the number of such compositions.

  Since $p \nmid l = \deg g$, the uniqueness of tame decompositions
  implies that the composition map $(g, w \circ h) \mapsto g \circ w
  \circ h$ is injective. To estimate from below the number $N$ of $w
  \circ h$, we use \ref{th:decom} with $u=p^{d}$, $a=m/lp^{d}$,
  $k=m/l$, $ \tilde{m}=l \neq u$, $\mu = \gcd (u-1,l)$, and $r=
  (u-1)/\mu$.  (Here $\tilde{m}$ is the value called $m$ in
  \ref{th:decom}, whose name conflicts with the present value of $m$.)

  If $\mu = 1$, we obtain from \ref{th:decom-1}
  $$
  N \geq q^{l+m/l-2}(1-q^{-m/l})(1-q^{-1}(1+q^{-p+2}
  \frac{(1-q^{-1})^{2}}{1-q^{-p}})).
  $$

  If $\mu\neq 1$, \ref{th:decom-2} says that
  \begin{align*}
    N \geq \, & q^{l+m/l-2}\bigl((1-q^{-1}(1+q^{-p+2}
    \frac{(1-q^{-1})^{2}}{1-q^{-p}}))(1-q^{-m/l})\\
    & -q^{-m/l-r+2}
    \frac{(1-q^{-1})^{2}(1-q^{-r(\mu-1)})}{1-q^{-r}}
    (1+q^{-r(p-2)})\bigr).
  \end{align*}

  We compose these $w \circ h$ with a monic original $g$ of degree $l$
  on the left. This gives the lower bound $q^{l-1}N$ on $t$, as claimed.
\end{proof}

\begin{example}\label{ex:p2}
  We study the particular example $p=l=2$ and $m=6$, so that $n=12$.
  Special considerations yield a better bound than the general
  one. This will be used in a forthcoming work on counting
  decomposable polynomials.  Let $t = \#(D_{12,2} \cap D_{12,6}
 \setminus F[x^2])$. Then \ref{cor:ffchar-3} says that $t \leq q^{5}$. By
  coefficient comparison, we now find a better bound. Namely, we are
  looking for $g \circ h = g^{*} \circ h^{*}$ with $g, h, g^{*}, h^{*}
  \in \mathbb{F}_{q}[x]$ monic original of degrees $2,6,6,2$,
  respectively. (We have reversed the usual degrees of $g$, $h$ and
  $g^{*}$, $h^{*}$ for notational convenience.)  We write $h =
  \sum_{i}h_{i}x^{i}$, and similarly for the other polynomials.  Then
  we choose any $h_{2}, h_{4}, h_{5} \in \mathbb{F}_{q}$, and either
  $g_{1}$ arbitrary and $h_{1}= uh_{5}$, or $h_{1}$ arbitrary and
  $g_{1} = h_{5}(h_{1}+uhs)$, where $u = h_{5}^{4} + h_{5}^{2}
  h_{4}+h_{2}$. Furthermore, we set $h_{3}= h_{5}^{3}$ and $h_{1}^{*}=
  h_{5}$. Then the coefficients of $g^{*}$ are determined. If
  $g'(g^{*})'\neq 0$, then the above constitute a collision, and by
  comparing coefficients, one finds that these are all. Their number
  is at most $2q^{4}$, so that $t \leq 2q^{4}$.

For an explicit description of $g^{*}$, we set $u_{2}= h_{4} + h_{5}^{2}$.
In the first case, where $h_{1}= uh_{5}$, we have
$$
g^{*}= x^{6}+u_{2}^{2}x^{4}+g_{1}x^{3}+(u^{2}+u_{2}g_{1})x^{2}+g_{1}ux.
$$
In the second case, we have
$$
g^{*}= x^{6}+u_{2}^{2}x^{4}+h_{5}(h_{1}+uh_{5})x^{3}+(u_{2}h_{1}h_{5}+
uh_{2})x^{2}+h_{1}(h_{1}+uh_{5})x.
$$
In both cases, $g_{1}= g' \neq 0$ implies that $(g^{*})' \neq 0$.
\end{example}

%The lower bound in \ref{thm:FFC-1} shows that there are more
%polynomials in the intersection when $l^{2} \mid \dg$ than otherwise.
\cite{gie88}, Theorem 3.8, shows that there exist polynomials of
degree $n$ over a field of characteristic $p$ with super-polynomially
many decompositions, namely at least $\dg ^{\lambda \log n}$ many,
where $\lambda = (6 \log p)^{-1}$.

% ab Unterkapitel 4.1 bis Ende Kapitel 4 nach Anweisung von vzG
% gelöscht! Sie sind jetzt nur noch in
% V:\publications\dec\invariant1.tex enthalten! 7.Juli 2009

\section{Acknowledgements}

Many thanks go to Jaime Guti\'errez for alerting me to Umberto
Zannier's paper, and to Umberto Zannier for correcting a
misunderstanding. I appreciate the discussions with Arnaud Bodin,
Pierre D\`ebes, and Salah Najib about the topic, and in particular the
challenges that their work \cite{boddeb09} posed.

This work was supported by the B-IT Foundation and the Land
Nordrhein-Westfalen, some of its results are cited in
\cite{gat14}, and a final version
 has appeared in
  \emph{Finite Fields and Their Applications} {\bf 27} (2014), pages 41--71.
  \copyright\ 2013 Elsevier Inc.
This published version
 contains an error in \ref{lem:propdivi}, which is corrected here.

%\bibliographystyle{cc2}
%\bibliography{journals,refs,lncs}
% \cleardoublepage
% \input{rittappendix1}

%\section*{References}
%%% \bibliography{journals,refs,lncs}
\bibliography{\jobname}

%Inhalt der bbl-Datei am Ende hier rein kopieren
 \providecommand{\ymd}[3]{\csname @ifempty\endcsname{#1}{?#1/#2/#3?}{\csname
  @ifempty\endcsname{#2}{?#1/#2/#3?}{\csname
  @ifempty\endcsname{#3}{?#1/#2/#3?}{{\relax \day=#3\relax \month=#2\relax
  \year=#1\relax \number\day~\ifcase\month\or January\or February\or March\or
  April\or May\or June\or July\or August\or September\or October\or November\or
  December\fi \space\ifnum\year>0\relax \number\year \else \csname
  count@\endcsname1\relax \expandafter\advance\csname count@\endcsname-\year
  \expandafter\number\csname count@\endcsname~BC\fi}}}}}
  \providecommand{\hide}[1]{.} \providecommand{\Hide}[1]{\unskip}
  \providecommand{\gobble}[1]{} \providecommand{\todo}[1]{\textbf{`?`?`?#1???}}
  \providecommand{\Name}[1]{#1} \providecommand{\Textgreek}[1]{\textgreek{#1}}
  \providecommand{\at}{\char64\relax}
  \providecommand{\cyr}{\PackageError{cyrillic}{Package not loaded. Use
  \string\usepackage{cyrillic} to define \string\cyr\space appropriately}{}
  \gdef\cyr{\def\cprime{c'}{\bf ?cyr?}}\cyr}
  \makeatletter\protected@write\@auxout{}{\string
  \gdef\string\abbr{\string\csname\space
  @gobble\string\endcsname}}\gdef\abbr{}\makeatother
  \makeatletter\protected@write\@auxout{}{\string
  \gdef\string\bibliographyonly{\string\protect\string\abbr}}\gdef\bibliographyonly{\protect\abbr}\makeatother
  \makeatletter\protected@write\@auxout{}{\string
  \gdef\string\bodyonly{}}\gdef\bodyonly#1{}\makeatother
\begin{thebibliography}{28}
\providecommand{\natexlab}[1]{#1}
\providecommand{\url}[1]{\texttt{#1}}
\providecommand{\urlprefix}{URL }
\providecommand{\selectlanguage}[1]{\relax}

\bibitem[{Barton \& Zippel(1985)}]{barzip85}
\textsc{David~R. Barton} \& \textsc{Richard Zippel} (1985).
\newblock {Polynomial} {Decomposition} {Algorithms}.
\newblock \emph{Journal of Symbolic Computation} \textbf{1}, 159--168.

\bibitem[{Barton \& Zippel(1976)}]{barzip76}
\textsc{David~R. Barton} \& \textsc{Richard~E. Zippel} (1976).
\newblock A Polynomial Decomposition Algorithm.
\newblock In \emph{Proceedings of the third ACM Symposium on Symbolic and
  Algebraic Computation}, \textsc{Richard~D. Jenks}, editor, 356--358. ACM
  Press, Yorktown Heights, New York, United States.
\newblock \urlprefix\url{http://dx.doi.org/10.1145/800205.806356}.

\bibitem[{Beardon \& Ng(2000)}]{beang00}
\textsc{A.~F. Beardon} \& \textsc{T.~W. Ng} (2000).
\newblock On Ritt's Factorization of Polynomials.
\newblock \emph{Journal of the London Mathematical Society} \textbf{62},
  127--138.
\newblock
  \urlprefix\url{http://journals.cambridge.org/action/displayAbstract?fromPage=online&aid=58787}.

\bibitem[{Binder(1995)}]{bin95a}
\textsc{Franz Binder} (1995).
\newblock Characterization of polynomial prime bidecompositions: a simplified
  proof.
\newblock \emph{Contributions to General Algebra} \textbf{9}, 61--72.
\newblock
  \urlprefix\url{http://citeseerx.ist.psu.edu/viewdoc/download?doi=10.1.1.23.5819&rep=rep1&type=pdf}.

\bibitem[{Blankertz \emph{et~al.}(2013)Blankertz, von~zur Gathen \&
  Ziegler}]{blagat13}
\textsc{Raoul Blankertz}, \textsc{Joachim von~zur Gathen} \& \textsc{Konstantin
  Ziegler} (2013).
\newblock Compositions and collisions at degree $p^2$.
\newblock \emph{Journal of Symbolic Computation} \textbf{59}, 113--145.
\newblock ISSN 0747-7171.
\newblock \urlprefix\url{http://dx.doi.org/10.1016/j.jsc.2013.06.001}.
\newblock Extended abstract in \bgroup\em Proceedings of the 2012 International
  Symposium on Symbolic and Algebraic Computation ISSAC2012, {\rm Grenoble,
  France}\egroup\ (2012), 91--98.

\bibitem[{Bodin \emph{et~al.}(2009)Bodin, D{\`{e}}bes \& Najib}]{boddeb09}
\textsc{Arnaud Bodin}, \textsc{Pierre D{\`{e}}bes} \& \textsc{Salah Najib}
  (2009).
\newblock Indecomposable polynomials and their spectrum.
\newblock \emph{Acta Arithmetica} \textbf{139(1)}, 79--100.

\bibitem[{Corrales-Rodrig{\'{a}}{\~{n}}ez(1990)}]{cor90}
\textsc{Capi Corrales-Rodrig{\'{a}}{\~{n}}ez} (1990).
\newblock A note on Ritt's Theorem on decomposition of polynomials.
\newblock \emph{Journal of Pure and Applied Algebra} \textbf{68}(3), 293--296.
\newblock
  \urlprefix\url{http://www.sciencedirect.com/science/journal/00224049}.

\bibitem[{Dorey \& Whaples(1974)}]{dorwha74}
\textsc{F.~Dorey} \& \textsc{G.~Whaples} (1974).
\newblock {Prime} and {Composite} {Polynomials}.
\newblock \emph{Journal of Algebra} \textbf{28}, 88--101.
\newblock \urlprefix\url{http://dx.doi.org/10.1016/0021-8693(74)90023-4}.

\bibitem[{von~zur Gathen(1990{\natexlab{a}})}]{gat90c}
\textsc{Joachim von~zur Gathen} (1990{\natexlab{a}}).
\newblock {Functional} {Decomposition} of {Polynomials}: the {Tame} {Case}.
\newblock \emph{Journal of Symbolic Computation} \textbf{9}, 281--299.
\newblock \urlprefix\url{http://dx.doi.org/10.1016/S0747-7171(08)80014-4}.

\bibitem[{von~zur Gathen(1990{\natexlab{b}})}]{gat90d}
\textsc{Joachim von~zur Gathen} (1990{\natexlab{b}}).
\newblock {Functional} {Decomposition} of {Polynomials}: the {Wild} {Case}.
\newblock \emph{Journal of Symbolic Computation} \textbf{10}, 437--452.
\newblock \urlprefix\url{http://dx.doi.org/10.1016/S0747-7171(08)80054-5}.

\bibitem[{von~zur Gathen(2010{\natexlab{a}})}]{gat10f}
\textsc{Joachim von~zur Gathen} (2010{\natexlab{a}}).
\newblock Counting decomposable multivariate polynomials.
\newblock \emph{Applicable Algebra in Engineering, Communication and Computing}
  \textbf{22}(3), 165--185.
\newblock \urlprefix\url{http://dx.doi.org/10.1007/s00200-011-0141-9}.

\bibitem[{von~zur Gathen(2010{\natexlab{b}})}]{gat10}
\textsc{Joachim von~zur Gathen} (2010{\natexlab{b}}).
\newblock {Shift-invariant} polynomials and {R}itt's {S}econd {T}heorem.
\newblock \emph{Contemporary Mathematics} \textbf{518}, 161--184.

\bibitem[{von~zur Gathen(2013)}]{gat13}
\textsc{Joachim von~zur Gathen} (2013).
\newblock Lower bounds for decomposable univariate wild polynomials.
\newblock \emph{Journal of Symbolic Computation} \textbf{50}, 409--430.
\newblock \urlprefix\url{http://dx.doi.org/10.1016/j.jsc.2011.01.008}.

\bibitem[{von~zur Gathen(2014)}]{gat14}
\textsc{Joachim von~zur Gathen} (2014).
\newblock Counting decomposable univariate polynomials.
\newblock \emph{Combinatorics, Probability and Computing, Special Issue, to
  appear} \textbf{\unskip\null}.
\newblock Extended abstract in \bgroup\em Proceedings of the 2009 International
  Symposium on Symbolic and Algebraic Computation ISSAC2009, {\rm Seoul,
  Korea}\egroup\ (2009), 359--366. Preprint (2008) available at
  \url{http://arxiv.org/abs/0901.0054}.

\bibitem[{von~zur Gathen \emph{et~al.}(2013)von~zur Gathen, Viola \&
  Ziegler}]{gatvio13}
\textsc{Joachim von~zur Gathen}, \textsc{Alfredo Viola} \& \textsc{Konstantin
  Ziegler} (2013).
\newblock Counting reducible, powerful, and relatively irreducible multivariate
  polynomials over finite fields.
\newblock \emph{SIAM Journal on Discrete Mathematics} \textbf{27}(2), 855--891.
\newblock \urlprefix\url{http://dx.doi.org/10.1137/110854680}.
\newblock Also available at \url{http://arxiv.org/abs/0912.3312}. Extended
  abstract in \bgroup\em Proceedings of LATIN~2010, {\rm Oaxaca,
  Mexico}\egroup\ (2010).

\bibitem[{Giesbrecht(1988)}]{gie88}
\textsc{Mark~William Giesbrecht} (1988).
\newblock Complexity Results on the Functional Decomposition of Polynomials.
\newblock Technical Report 209/88, University of Toronto, Department of
  Computer Science, Toronto, Ontario, Canada.
\newblock Available as \url{http://arxiv.org/abs/1004.5433}.

\bibitem[{Kozen \& Landau(1989)}]{kozlan89}
\textsc{Dexter Kozen} \& \textsc{Susan Landau} (1989).
\newblock Polynomial Decomposition Algorithms.
\newblock \emph{Journal of Symbolic Computation} \textbf{7}, 445--456.
\newblock An earlier version was published as Technical Report 86-773, Cornell
  University, Department of Computer Science, Ithaca, New York, 1986.

\bibitem[{Levi(1942)}]{lev42}
\textsc{H.~Levi} (1942).
\newblock Composite {Polynomials} with coefficients in an arbitrary {Field} of
  characteristic zero.
\newblock \emph{American Journal of Mathematics} \textbf{64}, 389--400.

\bibitem[{Lidl \& Mullen(1993)}]{lidmul93}
\textsc{R.~Lidl} \& \textsc{G.~L. Mullen} (1993).
\newblock When Does a Polynomial over a Finite Field Permute the Elements of
  the Field?, {II}.
\newblock \emph{The American Mathematical Monthly} \textbf{100}, 71--74.

\bibitem[{Lidl \emph{et~al.}(1993)Lidl, Mullen \& Turnwald}]{lidmul93a}
\textsc{R.~Lidl}, \textsc{G.~L. Mullen} \& \textsc{G.~Turnwald} (1993).
\newblock \emph{Dickson polynomials}.
\newblock Number~65 in Pitman Monographs and Surveys in Pure and Applied
  Mathematics. Longman Scientific \& Technical.
\newblock ISBN 0-582-09119-5.

\bibitem[{Ritt(1922)}]{rit22}
\textsc{J.~F. Ritt} (1922).
\newblock Prime and Composite Polynomials.
\newblock \emph{Transactions of the American Mathematical Society} \textbf{23},
  51--66.
\newblock \urlprefix\url{http://www.jstor.org/stable/1988911}.

\bibitem[{Schinzel(1982)}]{sch82c}
\textsc{Andrzej Schinzel} (1982).
\newblock \emph{Selected Topics on Polynomials}.
\newblock Ann Arbor; The University of Michigan Press.
\newblock ISBN 0-472-08026-1.

\bibitem[{Schinzel(2000)}]{sch00c}
\textsc{Andrzej Schinzel} (2000).
\newblock \emph{Polynomials with special regard to reducibility}.
\newblock Cambridge University Press, Cambridge, UK.
\newblock ISBN 0521662257.

\bibitem[{Tortrat(1988)}]{tor88a}
\textsc{Pierre Tortrat} (1988).
\newblock Sur la composition des polyn{\^{o}}mes.
\newblock \emph{Colloquium Mathematicum} \textbf{55}(2), 329--353.

\bibitem[{Turnwald(1995)}]{tur95}
\textsc{Gerhard Turnwald} (1995).
\newblock On Schur's Conjecture.
\newblock \emph{Journal of the Australian Mathematical Society, Series~A}
  \textbf{58}, 312--357.
\newblock
  \urlprefix\url{http://anziamj.austms.org.au/JAMSA/V58/Part3/Turnwald.html}.

\bibitem[{Williams(1971)}]{wil71}
\textsc{Kenneth~S. Williams} (1971).
\newblock Note on Dickson's permutation polynomials.
\newblock \emph{Duke Mathematical Journal} \textbf{38}, 659--665.
\newblock
  \urlprefix\url{http://mathstat.carleton.ca/~williams/papers/pdf/041.pdf}.

\bibitem[{Zannier(1993)}]{zan93}
\textsc{U{\hide{mberto}}~Zannier} (1993).
\newblock {Ritt's} {Second} {Theorem} in arbitrary characteristic.
\newblock \emph{Journal f{\"{u}}r die reine und angewandte Mathematik}
  \textbf{445}, 175--203.

\bibitem[{Zieve \& M{\"{u}}ller(2008)}]{ziemue08}
\textsc{Michael~E. Zieve} \& \textsc{Peter M{\"{u}}ller} (2008).
\newblock On Ritt's Polynomial Decomposition Theorems.
\newblock arXiv:0807.3578.

\end{thebibliography}

\end{document}